\documentclass[12pt]{amsart}
\usepackage{epsfig,color}
\usepackage{blindtext}
\usepackage{graphicx}
\usepackage{enumitem}
\usepackage{url}
\usepackage{amssymb}
\usepackage{graphicx,import}
\usepackage{comment}
\usepackage{esint}
\usepackage{xcolor}
\usepackage{mathtools}
\usepackage{comment}
\usepackage{tikz}

\usepackage[margin = 1in] {geometry}

\usepackage{hyperref}
\usepackage{dsfont}

\newtheorem{theorem}{Theorem}[section]

\newtheorem{proposition}[theorem]{Proposition}
\newtheorem{lemma}[theorem]{Lemma}
\newtheorem{corollary}[theorem]{Corollary}

\theoremstyle{definition}

\newtheorem*{remark*}{Remark}
\newtheorem*{acknowledgements}{Acknowledgements}
\newtheorem{remark}[theorem]{Remark}

\numberwithin{equation}{section}

\newcommand{\cA}{\mathcal{A}}

\newcommand{\cH}{\mathcal{H}}
\newcommand{\cS}{\mathcal{S}}

\newcommand{\R}{\mathbb{R}}

\newcommand{\M}{\mathbb{M}}
\newcommand{\ct}{\mathrm{ct}}

\DeclareMathOperator{\dist}{dist}

\DeclareMathOperator{\length}{length}
\DeclareMathOperator{\E}{E}
\newcommand{\eps}{\varepsilon}

\title{Existence of closed embedded curves of constant curvature via min-max}

\author{Lorenzo Sarnataro}
\address{Department of Mathematics, Princeton University, Princeton, NJ 08540, USA}
\email{lorenzos@princeton.edu}

\author{Douglas Stryker}
\address{Department of Mathematics, Princeton University, Princeton, NJ 08540, USA}
\email{dstryker@princeton.edu}

\begin{document}

\maketitle


\begin{abstract}
    We find conditions under which Almgren-Pitts min-max for the prescribed geodesic curvature functional in a closed oriented Riemannian surface produces a closed embedded curve of constant curvature. In particular, we find a closed embedded curve of any prescribed constant curvature in any metric on $S^2$ with $1/8$-pinched Gaussian curvature.
\end{abstract}

\section{Introduction}

It is a foundational problem in symplectic geometry, dynamical systems, and variational methods to find closed curves of constant geodesic curvature in a closed oriented Riemannian surface. This problem was famously posed by Arnold (see \cite[Problem 1981-9]{arnold}), and it has inspired an abundance of research using techniques from many fields of mathematics (for a few examples, see \cite{GinzburgNew}, \cite{Ginzburg}, \cite{SchneiderS2}, \cite{AB}, \cite{DRZ}).

We focus on the problem of finding closed \emph{embedded} curves of constant curvature. From the dynamical perspective, these objects correspond to \emph{simple} closed orbits of the flow of electrons in a constant magnetic field. As a motivating problem, it is a long-standing conjecture of Novikov (see \cite[\S5]{novikov}) that every metric on the 2-sphere admits a closed embedded curve of constant curvature $c$ for every $c > 0$.

In this paper, we take a variational approach to the existence problem. Namely, we use min-max techniques to produce non-singular critical points of the functional
\[ \cA^c(\Omega) := \cH^1(\partial \Omega) - c\cH^2(\Omega). \]
The Euler-Lagrange equation for this functional demonstrates that any non-singular critical point is a domain bounded by a closed curve of curvature $c$.

Using the Almgren-Pitts min-max framework, as developed for the $\cA^c$ functional for curves in surfaces by \cite{ZhouZhuCurves} and \cite{KL}, we obtain the following general existence result. For convenience, we define\footnote{The function $\ct_k(r)$ appears naturally in comparison geometry, see \cite[\S6.4]{petersen}}
\[ \ct_k(r) := \begin{cases}
    \sqrt{k}\cot(r\sqrt{k}) & k > 0\\
    1/r & k = 0\\
    \sqrt{-k}\coth(r\sqrt{-k}) & k < 0.
\end{cases} \]

\begin{theorem}\label{thm:general_surface}
    Let $(M^2, g)$ be a smooth closed oriented Riemannian surface and $c > 0$. If
    \begin{equation}\label{eqn:formula_general_surface}
        c > \ct_{\min K_g}(\mathrm{inj}(M,g)),
    \end{equation}
    then there is a nontrivial closed embedded curve of constant curvature $c$ in $(M, g)$ with $\cA^c$ Morse index 1.
\end{theorem}

In \S\ref{sec:examples}, we develop examples that demonstrate the ways in which Theorem \ref{thm:general_surface} is sharp. For instance, we show that Theorem \ref{thm:general_surface} is sharp for flat tori. It is also sharp for closed hyperbolic surfaces of large injectivity radius and large $c$. Finally, we exhibit a sphere that shows that Theorem \ref{thm:general_surface} is sharp for the method of proof.

We obtain a stronger result in positive ambient curvature. We summarize the result with a slightly weakened but simplified statement here. We refer the reader to (\ref{eqn:conditions}) for the full result, and to Figure \ref{fig:comparison_pos} for an illustration of the full result.

\begin{theorem}\label{thm:positive_curvature}
    Let $(S^2, g)$ be a smooth Riemannian 2-sphere with $K_g > 0$, and let $c > 0$. If
    \begin{equation}\label{eqn:formula_positive_curvature}
        \min K_g \geq \frac{1}{8}\max K_g \ \ \mathrm{or}\ \ c \geq \frac{1}{\pi}\sqrt{\max K_g},
    \end{equation}
    then there is a nontrivial closed embedded curve of constant curvature $c$ in $(S^2, g)$ with $\cA^c$ Morse index 1.
\end{theorem}

We expect that Theorem \ref{thm:positive_curvature} is not sharp, although it is not clear how to significantly improve upon the methods of this paper. Indeed, since our result surpasses the long-standing and seemingly natural \emph{quarter-pinching} threshold (see \cite{SchneiderS2}), we view Theorem \ref{thm:positive_curvature} as strong evidence in favor of the conjecture of Novikov in positive curvature.

\subsection{Comparison to previous results}
In the case of the 2-sphere (i.e.\ the conjecture of Novikov), the only successful program for producing closed embedded curves of constant curvature is a degree theory developed by \cite{SchneiderS2}. Let $(S^2, g)$ be a smooth Riemannian 2-sphere. Below we compare the conditions under which we and \cite{SchneiderS2} produce a closed embedded curve of constant curvature $c$.

\emph{Case 1: $\min K_g > 0$}. Suppose (without loss of generality) that $\max K_g = 1$.
\begin{itemize}
    \item \cite{SchneiderS2}: $\min K_g > 1/4$ or $c \geq 1/2$.
    \item Our result: $\min K_g \geq 1/8$ or $c \geq 1/\pi$.
\end{itemize}
In fact, our result is even stronger than stated. Figure \ref{fig:comparison_pos} illustrates precisely our improvement. We note that we only need to assume $\min K_g \geq 0.1167$, which is smaller than $1/8$.

\begin{remark}
    \cite{RSpositive} use the degree theory approach to show that for any $(S^2, g)$ with $K_g > 0$, there is an (inexplicit) $\eps = \eps(S^2, g) > 0$ so that there is a closed embedded curve of curvature $c$ for all $c \leq \eps$. While our methods do not apply to the case of small $c$ in general, we obtain an \emph{explicit} lower bound for $\eps(S^2, g)$ when $\min K_g \geq \frac{1}{16}\max K_g$ (see Figure \ref{fig:comparison_pos}).
\end{remark}

\emph{Case 2: $\min K_g = 0$}.
\begin{itemize}
    \item \cite{SchneiderS2}: $c \geq \frac{\pi}{2}\,\mathrm{inj}(S^2, g)^{-1}$.
    \item Our result: $c > \mathrm{inj}(S^2, g)^{-1}$.
\end{itemize}
Hence, we improve by a factor of $\pi/2$.

\emph{Case 3: $\min K_g < 0$}. Suppose (without loss of generality) that $\min K_g = -1$.
\begin{itemize}
    \item \cite{SchneiderS2}: $c \geq \frac{\pi}{2}\,\mathrm{inj}(S^2, g)^{-1}(1 + \frac{\mathrm{Area}(S^2, g)}{2\pi})$.
    \item Our result: $c > \coth(\mathrm{inj}(S^2, g))$.
\end{itemize}
Figure \ref{fig:comparison_neg} illustrates our improvement.\footnote{By a simple area estimate (see \cite[\S6.1]{CK}), we have $\mathrm{Area}(S^2, g) \geq \mathrm{inj}(S^2, g)^2$. So at best, \cite{SchneiderS2} applies when $c > \frac{\pi}{2}\,\mathrm{inj}(S^2,g)^{-1}(1+\frac{\mathrm{inj}(S^2,g)^2}{2\pi})$, which is what we plot in Figure \ref{fig:comparison_neg}.}

\begin{figure}
    \begin{minipage}{0.49\textwidth}
    \centering
    \begin{tikzpicture}[scale=6,domain=0:0.5]
        \draw (0,0) -- (.125, 0) node[below] {$\frac{1}{8}$};
        \draw (.125,0) -- (.25,0) node[below] {$\frac{1}{4}$};
        \draw[->] (.25,0) -- (.5,0) node[right] {$\min K_g$};
        \draw (0,0) -- (0,1/pi) node[left] {$\frac{1}{\pi}$};
        \draw (0,1/pi) -- (0,.5) node[left] {$\frac{1}{2}$};
        \draw[->] (0,.5) -- (0,.65) node[above] {$c$};
        \draw[dashed] (.125,0) -- (.125,1/pi);
        \draw[dashed] (0,1/pi) -- (.125,1/pi);
        
        \draw[color=red] (0,0.5) -- (0.25,0.5) ;
        \draw[color=red] (0.25,0) -- (0.25,0.5) node[right]{\cite{SchneiderS2}};
    
        \draw[color=blue] (0,1/pi) -- (0.1/4, 0.291689)
            -- (0.2/4, 0.264142)
            -- (0.3/4, 0.235599)
            -- (0.4/4, 0.205981)
            -- (.1167, .1856) node[right] {us}
            -- (.1165, .175)
            -- (.1158, .1625)
            -- (.1147, .15)
            -- (.1132, .1375)
            -- (.1112, .125)
            -- (.1088, .1125)
            -- (.1059, .1)
            -- (.1026, .0875)
            -- (.0987, .075)
            -- (.0943, .0625)
            -- (.0893, .05)
            -- (.0838, .0375)
            -- (.0775, .025)
            -- (.0705, .0125)
            -- (1/16, 0);
        \draw (1/16, 0) node[below] {$\frac{1}{16}$};
    \end{tikzpicture}
    \caption{$0<K_g\leq 1$.}\label{fig:comparison_pos}
    \end{minipage}
    \begin{minipage}{0.49\textwidth}
    \centering
    \begin{tikzpicture}[scale=.39,domain=0:10]
        \draw (0,0) -- (5,0) node[below] {5};
        \draw[->] (5,0) -- (10,0) node[right] {$\mathrm{inj}(S^2, g)$};
        \draw (0,0) -- (0,1) node[left] {1};
        \draw (0,1) -- (0,5) node[left] {5};
        \draw[->] (0,5) -- (0,10.6) node[above] {$c$};
        \draw[dashed] (0, 1) -- (10, 1);

        \draw[color=red] (0.15770, 10) -- (0.17539, 9);
        \draw[color=red] (0.17539, 9) -- (0.19757, 8);
        \draw[color=red] (0.19757, 8) -- (0.22623, 7);
        \draw[color=red] (0.22623, 7) -- (0.26472, 6);
        \draw[color=red] (0.26472, 6) -- (0.31926, 5);
        \draw[color=red] (0.31926, 5) -- (0.40284, 4);
        \draw[color=red] (0.40284, 4) -- (0.46419, 3.5);
        \draw[color=red] (0.46419, 3.5) -- (0.54869, 3);
        \draw[color=red] (0.54869, 3) -- (0.67371, 2.5);
        \draw[color=red] (0.67371, 2.5) -- (0.88282, 2);
        \draw[color=red] (0.88282, 2) -- (1, 1.8208);
        \draw[color=red] (1, 1.8208) -- (1.25, 1.56914);
        \draw[color=red] (1.25, 1.56914) -- (1.5, 1.4222);
        \draw[color=red] (1.5, 1.4222) -- (1.75, 1.3351);
        \draw[color=red] (1.75, 1.3351) -- (2, 1.2854);
        \draw[color=red] (2, 1.2854) -- (2.5, 1.25332);
        \draw[color=red] (2.5, 1.25332) -- (3, 1.2736);
        \draw[color=red] (3, 1.2736) -- (3.5, 1.3238);
        \draw[color=red] (3.5, 1.3238) -- (4, 1.3927);
        \draw[color=red] (4, 1.3927) -- (5, 1.56416);
        \draw[color=red] (5, 1.56416) -- (6, 1.7618);
        \draw[color=red] (6, 1.7618) -- (7, 1.9744);
        \draw[color=red] (7, 1.9744) -- (8, 2.19635);
        \draw[color=red] (8, 2.19635) -- (9, 2.42453);
        \draw[color=red] (9, 2.42453) -- (10, 2.65708) node[above] {\cite{SchneiderS2}};

        \draw[color=blue] (0.10034, 10) -- (0.11157, 9);
        \draw[color=blue] (0.11157, 9) -- (0.12566, 8);
        \draw[color=blue] (0.12566, 8) -- (0.14384, 7);
        \draw[color=blue] (0.14384, 7) -- (0.16824, 6);
        \draw[color=blue] (0.16824, 6) -- (0.20273, 5);
        \draw[color=blue] (0.20273, 5) -- (0.25541, 4);
        \draw[color=blue] (0.25541, 4) -- (0.29389, 3.5);
        \draw[color=blue] (0.29389, 3.5) -- (0.34657, 3);
        \draw[color=blue] (0.34657, 3) -- (0.42365, 2.5);
        \draw[color=blue] (0.42365, 2.5) -- (0.54931, 2);
        \draw[color=blue] (0.54931, 2) -- (.625, 1.8031);
        \draw[color=blue] (.625, 1.8031) -- (.75, 1.5744);
        \draw[color=blue] (.75, 1.5744) -- (.875, 1.4206);
        \draw[color=blue] (.875, 1.4206) -- (1, 1.3130);
        \draw[color=blue] (1, 1.3130) -- (1.25, 1.1789);
        \draw[color=blue] (1.25, 1.1789) -- (1.5, 1.1048);
        \draw[color=blue] (1.5, 1.1048) -- (2, 1.0373);
        \draw[color=blue] (2, 1.0373) -- (3, 1.0050);
        \draw[color=blue] (3, 1.0050) -- (4, 1.0007);
        \draw[color=blue] (4, 1.0007) -- (5, 1.0001);
        \draw[color=blue] (5, 1.0001) -- (6, 1);
        \draw[color=blue] (6, 1) -- (10, 1) node[above] {us};
    \end{tikzpicture}
    \caption{$\min K_g = -1$.}\label{fig:comparison_neg}
    \end{minipage}
\end{figure}

\begin{remark}
    In fact, the degree theory approach produces \emph{two} distinct solutions in each of the above cases (due to topological considerations on $S^2$). We note that part of our argument can be used in combination with the degree theory of \cite{SchneiderS2} to also produce two distinct solutions under the assumptions of Theorem \ref{thm:general_surface} when the surface is $S^2$ (see Remark \ref{rem:schneider} for more details).
\end{remark}

\begin{remark}
    The degree theory approach also works to produce curves of prescribed (nonconstant) geodesic curvature, where the prescribing function is bounded from below by $c$. As in the above remark, a combination of part of our argument with \cite{SchneiderS2} can be used to produce two distinct closed embedded curves of prescribed (nonconstant) curvature in $S^2$ under the assumptions of Theorem \ref{thm:general_surface}.
\end{remark}

\begin{remark}
    Since the degree theory approach is not variational, the conclusion about the index is new in every case.
\end{remark}

As far as the authors are aware, there has been no positive work on the existence problem for \emph{embedded} curves in general smooth closed oriented Riemannian surfaces. The degree theory of Schneider is extended to surfaces of genus at least 2 in \cite{Schneiderhyp}, but over the class of \emph{Alexandrov embedded} curves, which need not be embedded. Hence, Theorem \ref{thm:general_surface} appears to be completely new for $M \neq S^2$.

\subsection{Strategy of proof}\label{ssec:strategy}
By the combined work of \cite{ZhouZhuCurves} and \cite{KL}, Almgren-Pitts min-max for $\cA^c$ produces a set $\Omega$ whose boundary is either:
\begin{enumerate}
    \item a closed embedded curve of constant curvature $c$ (with respect to $\Omega$),
    \item a closed almost-embedded curve of constant curvature $c$ (with respect to $\Omega$) with at least one self-touching point,
    \item a network of curves of constant curvature $c$ (with respect to $\Omega$) with one stationary node of degree 4.
\end{enumerate}
Our strategy is to find conditions under which (2) and (3) cannot occur.

In case (3), there are two possible configurations, according to whether $\Omega$ is disconnected (called ``Configuration 1'') or connected (called ``Configuration 2''). We refer the reader to Figure \ref{fig:configs} for a picture of these configurations. We observe that case (2) can be handled simultaneously with Configuration 2, so it suffices to rule out these two configurations.

\emph{Configuration 1:} We construct a competitor sweepout using the second variation of $\cA^c$ and a cut-and-paste procedure similar to \cite[\S4]{KL}, and conclude that any curve in Configuration 1 cannot achieve the first min-max width when $\min K_g + c^2 \geq 0$. Essentially, we find a graphical perturbation of a curve in Configuration 1 that decreases $\cA^c$, and from this new curve we construct a sweepout by resolving the self intersection in two different ways. Figure \ref{fig:sweepout} demonstrates the basic strategy of this construction (in particular, see steps (3) and (4) in Figure \ref{fig:sweepout}).

\begin{remark}
    The competitor sweepout we construct is very similar to the construction of \cite{CC} in the case of geodesics. We also observe that similar ideas are used in \cite{BWorkman}, although the setting of curves differs from the higher dimensional cases they consider in some important ways. Most importantly, the capacity argument of \cite{BWorkman} fails for curves.
\end{remark}

\emph{Configuration 2 (in general):} In this configuration, the second variation formula for $\cA^c$ has a bad term, so we cannot conclude as easily as in Configuration 1. Instead, we rule out this configuration \emph{only} using geometric considerations; namely, we completely ignore the fact that the configuration would have to achieve the min-max width.

We view $\Omega$ (which is connected in Configuration 2) as a surface with strictly $c$-convex boundary. We note that there are distinct points $p,\ q$ on the boundary that correspond to the same point in $M$ (i.e.\ a self touching point/node). Due to the convexity of the boundary, there is a length minimizing geodesic $\sigma$ in $\Omega$ from $p$ to $q$.
\begin{itemize}
    \item In $(M, g)$, $\sigma$ is a geodesic loop, so we have $\mathrm{length}(\sigma) \geq 2\,\mathrm{inj}(M, g)$.
    \item By a classical comparison geometry result of \cite{dekster}, $\mathrm{length}(\sigma)$ is bounded from above by the diameter of a closed curve of curvature $c$ in the model space of constant Gaussian curvature $\min K_g$.
\end{itemize}
These estimates contradict each other under the conditions of Theorem \ref{thm:general_surface}.

\begin{remark}\label{rem:schneider}
    The proof in \cite{SchneiderS2} also relies on the same geometric considerations, but \cite{SchneiderS2} uses a weaker upper bound for the length of $\sigma$. More precisely, \cite{SchneiderS2} observes that $\sigma$ is a stable curve and applies Bonnet-Myers when $K_g > 0$. By replacing this estimate with the estimate of \cite{dekster}, we obtain the same conclusions as \cite{SchneiderS2} under the assumptions of Theorem \ref{thm:general_surface} when $M = S^2$.
\end{remark}

\emph{Configuration 2 (in positive curvature):} In positive ambient curvature, we look for a non-graphical perturbation to construct a good competitor sweepout for curves in Configuration 2. In particular, we split the figure-eight curve $\partial \Omega$ into the union of two embedded loops $\gamma^1$ and $\gamma^2$, and construct the competitor sweepout in three steps.
\begin{enumerate}
    \item We use the curve shortening flow to construct a path between a point curve and $\gamma^1$.
    \item We choose a path of curves between $\gamma^1$ and $\gamma^2$ using a min-max procedure for length.
    \item We choose a path of curves between $\gamma^2$ and a point curve using constrained minimization for $\cA^c$.
\end{enumerate}
We refer the reader to Figure \ref{fig:length_sweepout} for an illustration of this sweepout. Since the curve shortening flow is length nonincreasing, the maximum length of curves in path (1) is $\mathrm{length}(\gamma^1)$. Using an estimate of \cite{dekster}, we obtain a good upper bound on the length of curves in path (2). Finally, \cite[Lemma 3.2]{KL} yields good $\cA^c$ bounds on path (3).

These estimates give an upper bound for the first $c$ min-max width. Combined with the lower bounds from the general case of Configuration 2, we deduce that a curve in Configuration 2 cannot achieve the first $c$ min-max width under the conditions of Theorem \ref{thm:positive_curvature} (or more specifically under conditions (\ref{eqn:conditions}).

\begin{remark}
    The are two non-sharp estimates used in the proof of Theorem \ref{thm:positive_curvature}. 
    \begin{enumerate}
        \item The lower bound for the length of $\gamma^1$ and $\gamma^2$: we ignore the fact that these curves have curvature $c$ and not 0.
        \item The upper bound for the $c$ min-max width: we ignore volume terms in the $\cA^c$ value of the competitor sweepout during the curve shortening flow piece (i.e. path (1)).
    \end{enumerate}
    While implementing these ideas would yield a slight improvement in the minimal curvature pinching for which we can solve the conjecture of Novikov, there is no hope of attacking the conjecture for all metrics of positive curvature along these lines because the sharper estimates would degenerate to our weaker estimates as $c \to 0$. Indeed, Figure \ref{fig:comparison_pos} demonstrates that our argument only works for $c \ll 1$ for metrics with at least $1/16$-curvature pinching.
\end{remark}

\subsection{Organization of the paper}
In \S\ref{sec:examples}, we develop some examples to examine the sharpness of Theorem \ref{thm:general_surface}. In \S\ref{sec:setup}, we recall the min-max setup of \cite{ZhouZhuCurves} and \cite{KL}, and we define Configurations 1 and 2. In \S\ref{sec:second_variation}, we compute the second variation of $\cA^c$ for curves in Configurations 1 and 2. In \S\ref{sec:cutandpaste}, we outline the competitor sweepout constructed from a good graphical perturbation, and we apply this construction in \S\ref{sec:cutandpaste_apply} to curves in Configurations 1 and 2. In \S\ref{sec:comparison}, we recall the comparison result of \cite{dekster} and apply it to rule out Configuration 2 in some cases, which concludes the proof of Theorem \ref{thm:general_surface}. In \S\ref{sec:positive_curvature}, we construct a competitor sweepout using a non-graphical perturbation of a curve in Configuration 2 to prove Theorem \ref{thm:positive_curvature}. Finally, in \S\ref{sec:index}, we prove the index bounds in Theorems \ref{thm:general_surface} and \ref{thm:positive_curvature}.

\begin{acknowledgements}
    The authors are grateful to their advisor Fernando Cod{\'a} Marques for several useful discussions. The authors are also grateful to Daniel Ketover for suggesting this subject and for his interest in this work.

    D.S.\ was supported by an NDSEG Fellowship.
\end{acknowledgements}

\section{Examples}\label{sec:examples}
We explore a few examples to characterize the sharpness of Theorem \ref{thm:general_surface}.

\subsection{Flat tori}
Let $\Lambda$ be a rank 2 lattice in $\R^2$. Without loss of generality we have $(0,0) \in \Lambda$ and (by rotating) $(0, 2r) \in \Lambda$ is the shortest nonzero vector in $\Lambda$ for some $r > 0$. Then $\R^2/\Lambda$ is a flat torus with injectivity radius $r$.

If $c > 1/r$, then $B_{1/c}(x_0)$ is compactly contained in the segment domain of $\R^2/\Lambda$ at $x_0 \in \R^2/\Lambda$, and its boundary is an embedded loop with constant curvature $c$.

Now suppose $c \leq 1/r$, and let $\gamma \subset \R^2/\Lambda$ be a closed curve of constant curvature $c$. Then $\gamma$ lifts to a circle $C$ of radius $1/c$ in $\R^2$. Since $1/c \geq r$, there are two points $(x_1, y_1),\ (x_2, y_2) \in C$ satisfying
\[ x_1 = x_2,\ y_1 = y_2 - 2r. \]
Hence, $(x_1, y_1)$ and $(x_2, y_2)$ descend to the same point in $\R^2/\Lambda$, so $\gamma$ is not embedded.

In these examples, there is a closed embedded curve of constant curvature $c$ if and only if
\[ c > \frac{1}{\mathrm{inj}(\R^2/\Lambda)}. \]
Hence, Theorem \ref{thm:general_surface} is sharp for flat tori.

\subsection{Hyperbolic surfaces}
We recall two facts about closed hyperbolic surfaces.

First, there is no closed curve of geodesic curvature 1. A simple proof of this fact is contained in \cite[Example 3.7]{Ginzburg}.

Second, there is a sequence of closed hyperbolic surfaces $(M_i, g_i)$ with $\mathrm{inj}(M_i, g_i) \to \infty$. There are several different constructions of such examples, for instance \cite[Theorem A]{Buser} (see also \cite{BS}).

For a sequence as above, we have
\[ \lim_{i \to \infty} \coth(\mathrm{inj}(M_i, g_i)) = 1. \]
Hence, Theorem \ref{thm:general_surface} is sharp for smooth closed hyperbolic surfaces of large injectivity radius and $c$ large.

\subsection{Spheres}
Consider the $C^{1,1}$ metric $g^*$ on $S^2$ given by gluing two radius 1 hemispheres in $\R^3$ to the boundary components of a cylinder of radius $1$ and length $L > \pi$ in $\R^3$. Then $0 \leq K_{g^*} \leq 1$ in the weak sense and $\mathrm{inj}(S^2, g^*) = \pi$\footnote{We can also perturb the metric $g^*$ to a smooth one that is arbitrarily close to satisfying these conditions.}.

By taking a circle of radius $\pi$ in the cylinder region, there is a closed almost embedded curve of constant curvature $1/\pi$ that has a self touching point in $(S^2, g^*)$.

Since our method in Theorem \ref{thm:general_surface} only finds cases where such configurations cannot exist for geometric reasons (without using further information such as the fact that it must achieve the min-max width), this example shows that Theorem \ref{thm:general_surface} is sharp for our method.

\section{Setup}\label{sec:setup}
Let $(M^2, g)$ be a smooth closed oriented Riemannian surface, and let $c > 0$.

Throughout the paper, we use the following terminology and convention for the curvature of curves. Let $\Omega\subset M$ be a set with smooth boundary, and let $N$ be the outward pointing normal unit vector field along $\partial \Omega$. We say $\partial \Omega$ has curvature $k : \partial \Omega \to \R$ with respect to $\Omega$ if $\kappa_{\partial \Omega} = kN$. To fix conventions, the circle of radius 1 has curvature 1 with respect to the ball of radius 1 in $\R^2$ (i.e. $\kappa_{\partial B_1}$ is the outward pointing unit normal vector field along $\partial B_1$).

For any set of finite perimeter $\Omega$, we define
\[ \cA^c(\Omega) := \cH^1(\partial \Omega) - c\cH^2(\Omega). \]

A family of sets of finite perimeter $\{\Omega_t\}_{t \in [0,1]}$ is a \emph{sweepout} if
\begin{itemize}
    \item $\Omega_0 = \varnothing$, $\Omega_1 = M$,
    \item $t \mapsto \mathbf{1}_{\Omega_t}$ is continuous in $L^1$,
    \item $t\mapsto\partial \Omega_t$ is continuous in the flat topology.
\end{itemize}

The \emph{first $c$ min-max width} of $(M, g)$ is
\[ W_c := \inf\left\{\sup_{t \in [0,1]} \cA^c(\Omega_t) \mid \{\Omega_t\}_{t\in[0,1]} \text{\ is\ a\ sweepout}\right\}. \]

The starting point for min-max in this setting is the following structure theorem of \cite{ZhouZhuCurves}.

\begin{theorem}\cite[Theorem 1.1]{ZhouZhuCurves}\label{thm:zz}
    There is a set of finite perimeter $\Omega \subset M$ and at most finitely many points $\{p_i\}_{i=1}^k \subset M$ so that
    \begin{itemize}
        \item $\cA^c(\Omega) = W_c$,
        \item away from $\{p_i\}_{i=1}^k$, $\partial \Omega$ is a smooth almost embedded curve of constant curvature $c$ with respect to $\Omega$,
        \item the density of $\partial \Omega$ at $p_i$ is an integer, and every tangent cone of $\partial \Omega$ at $p_i$ is a stationary geodesic network that is smooth away from the origin.
    \end{itemize}
\end{theorem}

We recall that a curve is \emph{almost embedded} if any self-intersection is tangential and the curve is locally a union of smooth segments that do not cross each other.

This structure theorem was refined using a cut-and-paste argument in \cite{KL}.

\begin{theorem}\cite[Theorem 1.3]{KL}\label{thm:kl}
    There is a set of finite perimeter $\Omega \subset M$ and at most one point $x_0 \in M$ so that
    \begin{itemize}
        \item $\cA^c(\Omega) = W_c$,
        \item away from $x_0$, $\partial \Omega$ is a smooth almost embedded curve of constant curvature $c$ with respect to $\Omega$,
        \item $\partial \Omega$ is the image of a $C^{1,1}$ immersion with a transverse self-intersection at $x_0$.
    \end{itemize}
\end{theorem}

We use this result as the starting point for our proof.

Let $\gamma: S^1 \to M$ be a $C^{1,1}$ arclength-parametrized immersion and let $\Omega \subset M$ be an open set so that
\begin{itemize}
    \item $\partial \Omega = \gamma(S^1)$,
    \item there are distinct points $t_0,\ t_1 \in S^1$ so that $\gamma\mid_{S^1 \setminus \{t_0, t_1\}}$ is a smooth almost embedding with curvature $c$ with respect to $\Omega$,
    \item there are disjoint neighborhoods $I_0,\ I_1 \subset S^1$ of $t_0$ and $t_1$ respectively so that $\gamma\mid_{I_0}$ and $\gamma\mid_{I_1}$ intersect transversely at $\gamma(t_0) = \gamma(t_1) = x_0 \in M$.
\end{itemize}

Let $\alpha \in (0, \pi)$ be the angle delimited by $\Omega$ at the self-intersection point $x_0$.

Let $N$ be the outward pointing unit normal vector field for $\Omega$ along $\gamma\mid_{S^1 \setminus \{t_0, t_1\}}$. Note that we have $\kappa_{\gamma} = cN$ on $S^1 \setminus \{t_0, t_1\}$. Let $\nu$ be a Lipschitz choice of unit normal vector field along $\gamma$. Let $I^+,\ I^- \subset S^1$ be the two components of $S^1 \setminus \{t_0, t_1\}$ so that $\nu\mid_{I^+} = N\mid_{I^+}$ and $\nu\mid_{I^-} = -N\mid_{I^-}$. Let $\gamma^{\pm} := \gamma\mid_{I^{\pm}}$.

\subsection{Configuration 1} $\gamma^+$ bounds a component of $\Omega$. Equivalently, $\Omega$ has two components. See the left image in Figure \ref{fig:configs} for an illustration.

\subsection{Configuration 2} $\gamma^+$ bounds a component of $M \setminus \Omega$. Equivalently, $\Omega$ is connected. See the right image in Figure \ref{fig:configs} for an illustration.

\begin{figure}
\centering
\begin{tikzpicture}
    \fill[gray!20!white] (-4, 0) .. controls (-3.5, .75) and (-2.5, 1) .. (-2, 1)
        .. controls (-1.5, 1) and (-1, .75) .. (-1, 0)
        .. controls (-1, -.75) and (-1.5, -1) .. (-2, -1)
        .. controls (-2.5, -1) and (-3.5, -.75) .. (-4, 0)
        .. controls (-4.5, .75) and (-5.5, 1) .. (-6, 1)
        .. controls (-6.5, 1) and (-7, .75) .. (-7, 0)
        .. controls (-7, -.75) and (-6.5, -1) .. (-6, -1)
        .. controls (-5.5, -1) and (-4.5, -.75) .. (-4, 0);
    \draw[thick] (-4, 0) .. controls (-3.5, .75) and (-2.5, 1) .. (-2, 1)
        .. controls (-1.5, 1) and (-1, .75) .. (-1, 0)
        .. controls (-1, -.75) and (-1.5, -1) .. (-2, -1)
        .. controls (-2.5, -1) and (-3.5, -.75) .. (-4, 0)
        .. controls (-4.5, .75) and (-5.5, 1) .. (-6, 1)
        .. controls (-6.5, 1) and (-7, .75) .. (-7, 0)
        .. controls (-7, -.75) and (-6.5, -1) .. (-6, -1)
        .. controls (-5.5, -1) and (-4.5, -.75) .. (-4, 0);
    \draw (-6, 0) node[left] {$\Omega$};
    \draw[ultra thick][->] (-7, 0) -- (-7.5, 0) node[above] {$\nu$};
    \draw[ultra thick][->] (-1, 0) -- (-1.5, 0) node[above] {$\nu$};
    \draw (-6, 1) node[above] {$\gamma^+$};
    \draw (-2, 1) node[above] {$\gamma^-$};
    \draw (-3.5,0) arc [start angle = 0, end angle = 50, radius = .5];
    \draw (-3.5,0) arc [start angle = 0, end angle = -50, radius = .5];
    \draw (-3.5, 0) node[right] {$\alpha$};
    \draw (-4, -2) node {Configuration 1};

    \fill[gray!20!white] (2, .75) .. controls (2.25, .75) and (2.75, .6) .. (3, .5)
        .. controls (3.25, .4) and (3.5, .3) .. (4, 0)
        .. controls (3.5, -.3) and (3.25, -.4) .. (3, -.5)
        .. controls (2.75, -.6) and (2.25, -.75) .. (2, -.75);
    \fill[gray!20!white] (6, .75) .. controls (5.75, .75) and (5.25, .6) .. (5, .5)
        .. controls (4.75, .4) and (4.5, .3) .. (4, 0)
        .. controls (4.5, -.3) and (4.75, -.4) .. (5, -.5)
        .. controls (5.25, -.6) and (5.75, -.75) .. (6, -.75);
    \draw[thick] (2, .75) .. controls (2.25, .75) and (2.75, .6) .. (3, .5)
        .. controls (3.25, .4) and (3.5, .3) .. (4, 0)
        .. controls (4.5, -.3) and (4.75, -.4) .. (5, -.5)
        .. controls (5.25, -.6) and (5.75, -.75) .. (6, -.75);
    \draw[thick] (2, -.75) .. controls (2.25, -.75) and (2.75, -.6) .. (3, -.5)
        .. controls (3.25, -.4) and (3.5, -.3) .. (4, 0)
        .. controls (4.5, .3) and (4.75, .4) .. (5, .5)
        .. controls (5.25, .6) and (5.75, .75) .. (6, .75);
    \draw[dashed] (2, .75) .. controls (1, .9) and (2, 1.25) .. (4, 1.25)
        .. controls (6, 1.25) and (7, .9) .. (6, .75);
    \draw[dashed] (2, -.75) .. controls (1, -.9) and (2, -1.25) .. (4, -1.25)
        .. controls (6, -1.25) and (7, -.9) .. (6, -.75);
    \draw (3, 0) node[left] {$\Omega$};
    \draw[ultra thick][->] (5, .5) -- (4.82, .95) node[left] {$\nu$};
    \draw[ultra thick][->] (5, -.5) -- (5.18, -.05) node[right] {$\nu$};
    \draw (2.75, .9) node {$\gamma^+$};
    \draw (2.75, -.85) node {$\gamma^-$};
    \draw (4.5,0) arc [start angle = 0, end angle = 30, radius = .5];
    \draw (4.5,0) arc [start angle = 0, end angle = -30, radius = .5];
    \draw (4.5, 0) node[right] {$\alpha$};
    \draw (4, -2) node {Configuration 2};
\end{tikzpicture}
\caption{An illustration of the two possible configurations.}
\label{fig:configs}
\end{figure}
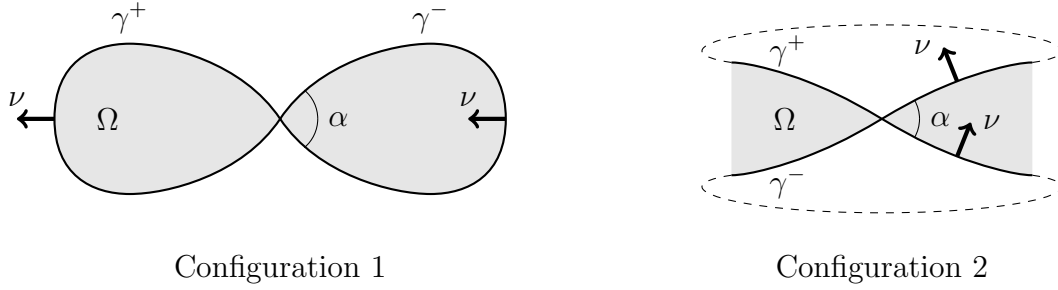

By allowing the angle $\alpha = \pi$ in Configuration 2, we note that Configuration 2 includes the case of a smooth closed almost embedded curve of constant curvature $c$ with at least one self-touching point. Hence, we will rule out the smooth almost embedded case when we rule out Configuration 2.

\section{Second Variation Formula}\label{sec:second_variation}
Let $X$ be a Lipschitz vector field along $\gamma$. Let $\Gamma : [0, \eps) \times S^1 \to M$ be defined as
\[ \Gamma(s, t) := \exp_{\gamma(t)}(sX(t)). \]
Note that $\Gamma$ satisfies
\begin{align*}
    \Gamma(0, t) & = \gamma(t)\\
    \partial_s\Gamma(0, t) & = X(t).
\end{align*}
We define
\begin{itemize}
    \item $\gamma_s(t) := \Gamma(s, t)$,
    \item $X_s(t) := \partial_s \Gamma(s, t)$,
    \item $T_s$ is the integer $2$-current given by $T_s := [\Omega] + \Gamma_{\#}[[0, s) \times S^1]$.
\end{itemize} 

Let $x_s \in M$ and $t_0(s),\ t_1(s) \in S^1$ so that
\begin{itemize}
    \item $\lim_{s \to 0} x_s = x_0$, the self-intersection point of $\gamma$,
    \item $s \mapsto x_s$ is a continuous curve in $M$,
    \item $t_0(0) = t_0$, $t_1(0) = t_1$,
    \item $s \mapsto t_0(s)$ and $s \mapsto t_1(s)$ are continuous curves in $S^1$,
    \item $\gamma_s^{-1}(x_s) = \{t_0(s), t_1(s)\}$.
\end{itemize}

Let $N_s$ be a Lipschitz unit normal vector field along $\gamma_s\mid_{S^1\setminus\{t_0(s), t_1(s)\}}$ so that $N_0 = N$ and $N_s(t)$ is continuous in $s$. Let $\nu_s$ be a Lipschitz choice of unit normal vector field along $\gamma_s$ so that $\nu_0 = \nu$ and $\nu_s(t)$ is continuous in $s$. Note that $\nu_s\mid_{I_s^+} = N_s\mid_{I_s^+}$ and $\nu_s\mid_{I_s^-} = -N_s\mid_{I_s^-}$ where $I_s^{\pm}$ are the components of $S^1 \setminus \{t_0(s), t_1(s)\}$. Let $\gamma_s^{\pm} := \gamma_s\mid_{I_s^{\pm}}$.

\begin{lemma}[First Variation]\label{lem:first_variation}
    \begin{equation}\label{eqn:first_variation}
        \frac{d}{ds} (\mathrm{length}(\gamma_s) - c\M(T_s)) = \int \frac{\langle X_s, \kappa_{\gamma_s} - cN_s\rangle}{|\partial_t \gamma_s|}\ dt.
    \end{equation}
\end{lemma}
\begin{proof}
    This formula is standard, see for instance \cite[Lemma 2.1]{BdCE}.
\end{proof}

\begin{lemma}[Second Variation]\label{lem:second_variation}
    Let $X(t) = \phi(t)\nu(t)$ for a Lipschitz function $\phi$ along $\gamma$.
    
    If $\gamma$ is in Configuration 1, then
    \begin{align}\label{eqn:second_variation_config1}
        \frac{d^2}{ds^2}\Big|_{s=0} & (\mathrm{length}(\gamma_s) - c\M(T_s))\\
        & = \int (|\phi'|^2 - (K_g + c^2)\phi^2)\ dt - \frac{2c}{\sin\alpha}\big((\phi(t_0)^2 + \phi(t_1)^2)\cos\alpha + 2\phi(t_0)\phi(t_1)\big).\notag
    \end{align}
    
    If $\gamma$ is in Configuration 2, then
    \begin{align}\label{eqn:second_variation_config2}
        \frac{d^2}{ds^2}\Big|_{s=0} & (\mathrm{length}(\gamma_s) - c\M(T_s))\\
        & = \int (|\phi'|^2 - (K_g + c^2)\phi^2)\ dt - \frac{2c}{\sin\alpha}\big((\phi(t_0)^2 + \phi(t_1)^2)\cos\alpha - 2\phi(t_0)\phi(t_1) \big).\notag
    \end{align}
\end{lemma}
\begin{proof}
    By standard arguments (see \cite[Proposition 2.5]{BdCE}) and the assumptions on $\gamma$, we have
    \begin{equation}\label{eqn:variation_length}
        \frac{d}{ds}\Big|_{s=0}\int \frac{\langle \phi \nu_s, \kappa_{\gamma_s}\rangle}{|\partial_t\gamma_s|} dt
        = \int (|\phi'|^2 - K_g\phi^2)\ dt,
    \end{equation}

    We compute
    \begin{equation*}
        \frac{d}{ds}\Big|_{s=0} \int \frac{\langle \phi \nu_s, N_s\rangle}{|\partial_t \gamma_s|}\ dt = \frac{d}{ds}\Big|_{s=0}\int_{I_s^+} \frac{\phi}{|\partial_t\gamma_s|}\ dt - \frac{d}{ds}\Big|_{s=0}\int_{I_s^-} \frac{\phi}{|\partial_t\gamma_s|}\ dt.
    \end{equation*}
    Without loss of generality we have $I^+ = (t_0, t_1)$ and $I^- = (t_1, t_0)$. Let $V := \partial_s\mid_{s=0} x_s$. By the standard first variation of length for curves with moving endpoints (see for instance \cite[Chapter 9]{docarmo}), we have
    \begin{align*}
        \frac{d}{ds}\Big|_{s=0} \int_{I_s^+} \frac{\phi}{|\partial_t\gamma_s|}\ dt
        & = c\int_{I^+} \phi^2\ dt -\phi(t_0)\langle V, \gamma'(t_0) \rangle + \phi(t_1)\langle V, \gamma'(t_1)\rangle\\
        \frac{d}{ds}\Big|_{s=0} \int_{I_s^-} \frac{\phi}{|\partial_t \gamma_s|}\ dt
        & = -c\int_{I^-} \phi^2\ dt + \phi(t_0)\langle V, \gamma'(t_0) \rangle - \phi(t_1)\langle V, \gamma'(t_1)\rangle.
    \end{align*}
    Hence, we have
    \begin{align}\label{eqn:variation_area}
        -c\frac{d}{ds}\Big|_{s=0} \int \frac{\langle \phi \nu_s, N_s\rangle}{|\partial_t \gamma_s|}\ dt =
        -\int c^2\phi^2\ dt +
        2c(\phi(t_0)\langle V, \gamma'(t_0) \rangle - \phi(t_1)\langle V, \gamma'(t_1)\rangle).
    \end{align}

    In Configuration 1, trigonometric considerations give
    \begin{equation}\label{eqn:variation_node_config1}
        V = \frac{1}{\sin\alpha}(-\phi(t_1)\gamma'(t_0) + \phi(t_0)\gamma'(t_1)),\ \ \text{and}\ \ 
        \langle \gamma'(t_0), \gamma'(t_1)\rangle = -\cos\alpha.
    \end{equation}
    
    In Configuration 2, trigonometric considerations give
    \begin{equation}\label{eqn:variation_node_config2}
        V = \frac{1}{\sin\alpha}(\phi(t_1)\gamma'(t_0) - \phi(t_0)\gamma'(t_1)),\ \ \text{and}\ \ 
        \langle \gamma'(t_0), \gamma'(t_1)\rangle = \cos\alpha.
    \end{equation}

    For clarity, we refer the reader to Figure \ref{fig:trig} to see (\ref{eqn:variation_node_config1}) and (\ref{eqn:variation_node_config2}).

    Together, Lemma \ref{lem:first_variation} with (\ref{eqn:variation_length}), (\ref{eqn:variation_area}), (\ref{eqn:variation_node_config1}), and (\ref{eqn:variation_node_config2}) conclude the proof.
\end{proof}

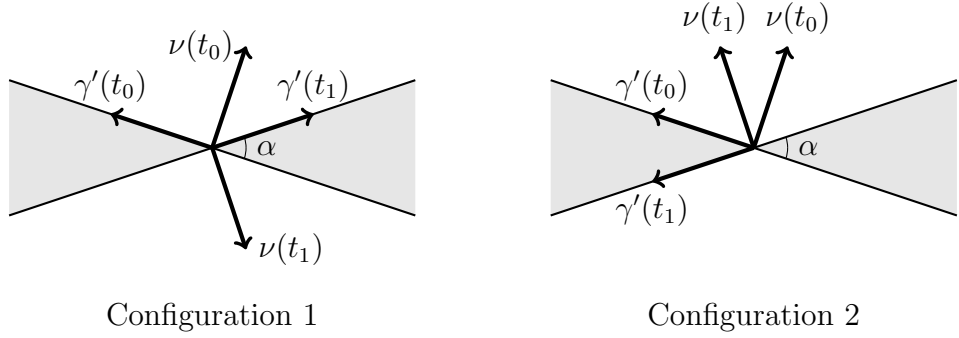
\begin{figure}
\centering
\begin{tikzpicture}[scale=.9]
    \fill[gray!20!white] (-7, 1) -- (-4, 0) -- (-7, -1);
    \fill[gray!20!white] (-1, 1) -- (-4, 0) -- (-1, -1);
    \draw[thick] (-7, 1) -- (-1, -1);
    \draw[thick] (-7, -1) -- (-1, 1);
    \draw (-3.5,0) arc [start angle = 0, end angle = 18.4, radius = .5];
    \draw (-3.5,0) arc [start angle = 0, end angle = -18.4, radius = .5];
    \draw (-3.5, 0) node[right] {$\alpha$};
    \draw[ultra thick][->] (-4,0) -- (-5.5, .5) node[above] {$\gamma'(t_0)$};
    \draw[ultra thick][->] (-4,0) -- (-2.5, .5) node[above] {$\gamma'(t_1)$};
    \draw[ultra thick][->] (-4, 0) -- (-3.5, 1.5) node[left] {$\nu(t_0)$};
    \draw[ultra thick][->] (-4, 0) -- (-3.5, -1.5) node[right] {$\nu(t_1)$};
    \draw (-4, -2.5) node {Configuration 1};

    \fill[gray!20!white] (1, 1) -- (4, 0) -- (1, -1);
    \fill[gray!20!white] (7, 1) -- (4, 0) -- (7, -1);
    \draw[thick] (1, 1) -- (7, -1);
    \draw[thick] (1, -1) -- (7, 1);
    \draw (4.5,0) arc [start angle = 0, end angle = 18.4, radius = .5];
    \draw (4.5,0) arc [start angle = 0, end angle = -18.4, radius = .5];
    \draw (4.5, 0) node[right] {$\alpha$};
    \draw[ultra thick][->] (4,0) -- (2.5, .5) node[above] {$\gamma'(t_0)$};
    \draw[ultra thick][->] (4,0) -- (2.5, -.5) node[below] {$\gamma'(t_1)$};
    \draw[ultra thick][->] (4, 0) -- (4.5, 1.5) node[above] {\ \ $\nu(t_0)$};
    \draw[ultra thick][->] (4, 0) -- (3.5, 1.5) node[above] {$\nu(t_1)$\ \ };
    \draw (4, -2.5) node {Configuration 2};
\end{tikzpicture}
\caption{Trigonometry considerations in Lemma \ref{lem:second_variation}.}
\label{fig:trig}
\end{figure}

\section{Cut-and-Paste Sweepout}\label{sec:cutandpaste}
Suppose $\gamma$, $\Omega$, $\gamma_s$, and $T_s$ are as in the previous section. In this section, we prove the following result.

\begin{theorem}\label{thm:cut_and_paste_sweepout}
    Suppose there is a function $\phi$ so that
    \[ \frac{d^2}{ds^2}\Big|_{s=0} (\mathrm{length}(\gamma_s) - c\M(T_s)) < 0. \]
    Then $\cA^c(\Omega)$ is not the first $c$-min-max width of $(M, g)$.
\end{theorem}

\subsection{Deletion of overlaps}
Before specifying the deletion of overlaps procedure, we require the following consequence of the isoperimetric inequality.

\begin{lemma}\label{lem:isoperimetric}
    There is a constant $\eta > 0$ (depending on $(M, g)$ and $c$) so that if $U \subset M$ is a set of finite perimeter with $\cH^2(U) \leq \eta$, then
    \[ \cA^c(U) \geq 0. \]
\end{lemma}
\begin{proof}
    Since $M$ is closed, there is a constant $c_1$ (depending on $(M, g)$) so that
    \[ \cH^1(\partial U) \geq c_1(\cH^2(U))^{1/2} \]
    for any set of finite perimeter $U$ with $\cH^2(U) \leq \frac{1}{2}\cH^2(M)$. Let $\eta = \min\{c_1^2/c^2, \cH^2(M)/2\}$. Then we have
    \begin{align*}
        \cA^c(U)
        & = \cH^1(\partial U) - c\cH^2(U)
        \geq \cH^2(U)^{1/2}(c_1 - c(\cH^2(U))^{1/2}) \geq 0,
    \end{align*}
    as desired.
\end{proof}

To eliminate overlaps, we simply delete one copy of the the multiplicity 2 piece of $T_s$. Namely, we define
\[ \Omega_s := \mathrm{spt}(T_s). \]
See Figure \ref{fig:delete} for an illustration of this procedure. We observe that $\Omega_0 = \Omega$ (up to a set of $\cH^2$-measure 0).

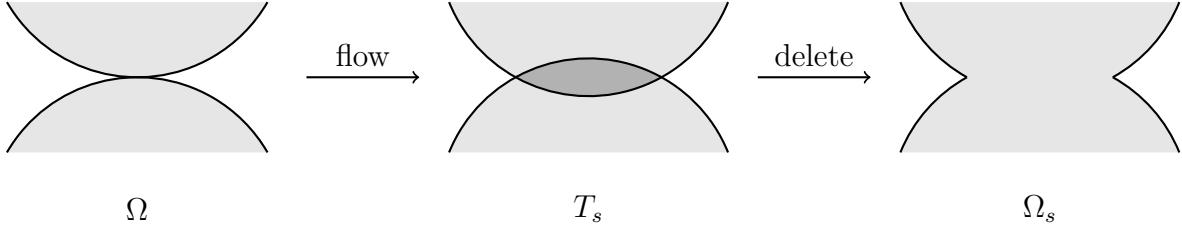
\begin{figure}
\centering
\begin{tikzpicture}
    \fill[gray!20!white] (-7.732, 1) arc [start angle = 210, end angle = 330, radius = 2];
    \fill[gray!20!white] (-7.732, -1) arc [start angle = 150, end angle = 30, radius = 2];
    \draw[thick] (-7.732, 1) arc [start angle = 210, end angle = 330, radius = 2];
    \draw[thick] (-7.732, -1) arc [start angle = 150, end angle = 30, radius = 2];
    \draw (-6, -1.75) node {$\Omega$};

    \draw[thick][->] (-3.75, 0) -- (-2.25, 0);
    \draw (-3, 0) node[above] {flow};

    \fill[gray!20!white] (-1.854, 1) arc [start angle = 202, end angle = 338, radius = 2];
    \fill[gray!20!white] (-1.854, -1) arc [start angle = 158, end angle = 22, radius = 2];
    \fill[gray!60!white] (-.968, 0) arc [start angle = 270-29, end angle = 299, radius = 2];
    \fill[gray!60!white] (-.968, 0) arc [start angle = 119, end angle = 90-29, radius = 2];
    \fill[gray!60!white] (-.96, -.02) -- (-.96, .02) -- (.96, .02) -- (.96, -.02);
    \draw[thick] (-1.854, 1) arc [start angle = 202, end angle = 338, radius = 2];
    \draw[thick] (-1.854, -1) arc [start angle = 158, end angle = 22, radius = 2];
    \draw (0, -1.75) node {$T_s$};

    \draw[thick][->] (2.25, 0) -- (3.75, 0);
    \draw (3, 0) node[above] {delete};

    \fill[gray!20!white] (4.146, 1) arc [start angle = 202, end angle = 338, radius = 2];
    \fill[gray!20!white] (4.146, -1) arc [start angle = 158, end angle = 22, radius = 2];
    \draw[thick] (4.146, 1) arc [start angle = 202, end angle = 270-29, radius = 2];
    \draw[thick] (7.854, 1) arc [start angle = 338, end angle = 299, radius = 2];
    \draw[thick] (4.146, -1) arc [start angle = 158, end angle = 119, radius = 2];
    \draw[thick] (7.854, -1) arc [start angle = 22, end angle = 90-29, radius = 2];
    \draw (6, -1.75) node {$\Omega_s$};
\end{tikzpicture}
\caption{Local picture of the deletion of overlaps procedure.}
\label{fig:delete}
\end{figure}

\begin{lemma}\label{lem:deletion_of_overlaps}
    There are $\eps > 0$, $\delta > 0$ depending on $\gamma$ and $\Omega$ so that if $|\phi| \leq 1$ and $0 \leq s \leq \eps$, then
    \begin{enumerate}
        \item $\partial \Omega_s \cap B_{\delta}(x_s)$ consists of two $C^{1,1}$ embeddings intersecting transversely at $x_s$, 
        \item $[\Omega_s] \llcorner B_{\delta}(x_s) = T_s \llcorner B_{\delta}(x_s)$,
        \item $\cA^c(\Omega_s) \leq \mathrm{length}(\gamma_s) - c\M(T_s)$.
    \end{enumerate}
\end{lemma}
\begin{proof}
    Conclusion (1) follows because transversality is an open condition.

    Conclusion (2) follows because there are no almost embedded points in a neighborhood of $x_0$, and in this neighborhood the fact that $T_s$ has multiplicity one is an open condition.

    It remains to show (3). Let $E : \R \times S^1 \to M$ be given by
    \[ E(s, t) := \exp_{\gamma(t)}(s\nu(t)). \]
    Let $U_s := E_{\#}[(-s, s) \times S^1]$. Since $E$ is Lipschitz, we have
    \[ \lim_{s \to 0} \M(U_s) = 0. \]
    In particular, we can find $\eps > 0$ small so that $\M(U_{\eps}) \leq \eta$.

    Since the almost embedded points of $\gamma$ have multiplicity 2, we have $T_s - [\Omega_s] = [O_s]$, where $O_s \subset M$ is an open set of finite perimeter satisfying $\cH^2(O_s) \leq \eta$.

    By Lemma \ref{lem:isoperimetric}, we have
    \begin{align*}
        \mathrm{length}(\gamma_s) - c\M(T_s)
        = \cA^c(\Omega_s) + \cH^1(\partial O_s) - c\cH^2(O_s) \geq \cA^c(\Omega_s),
    \end{align*}
    as desired.
\end{proof}

\subsection{Cut-and-paste scale}
We choose a uniform scale at which all steps in the cut-and-paste procedure work.

\begin{lemma}\label{lem:inj_c_radius}
    For a closed oriented Riemannian surface $(M^2, g)$ and any $c > 0$, there is a constant $r_c \in (0, \delta)$ (depending on $(M, g)$ and $c$) so that
    \begin{itemize}
        \item $\cH^2(B_{r_c}(x)) \leq \eta$ for all $x \in M$,
        \item $B_{\rho}(x)$ is strictly $c$-convex for all $x \in M$ and $\rho \leq r_c$,
        \item for any $x,\ y \in M$ with $d(x, y) \leq 2r_c$, there is a unique geodesic from $x$ to $y$ contained in $B_{2r_c}(x)$,
        \item for any $x,\ y \in M$ with $d(x, y) \leq 2r_c$, there is a unique curve $\sigma$ from $x$ to $y$ in $B_{2r_c}(x)$ with constant curvature $c$ so that $(\sigma'(t), \kappa_{\sigma}(t))$ is a positive basis,
        \item for any $x,\ y \in M$ with $d(x, y) \leq 2r_c$, there is a unique curve $\sigma$ from $x$ to $y$ in $B_{2r_c}(x)$ with constant curvature $c$ so that $(\sigma'(t), \kappa_{\sigma}(t))$ is a negative basis.
    \end{itemize}
\end{lemma}
\begin{proof}
    The result follows from standard ODE theory and the compactness of $M$, analogous to the case of the geodesic flow (for instance, see \cite[Chapter 3.2]{docarmo}).
\end{proof}

\subsection{Cut-and-paste procedure}\label{ssec:cut-paste}
For any $0 \leq r \leq r_c$ and any $0 \leq s \leq \eps$, we define sets $\Omega_{s,r}^+$ and $\Omega_{s,r}^-$ as follows.

Let $\{\Omega^1, \Omega^2,\Omega^3,\Omega^4\}$ be the components of $B_r(x_s) \setminus \partial \Omega_s$. We choose the labels so that $\Omega^1,\ \Omega^3 \subset \Omega_s$. Let $\{p_1, p_2, p_3, p_4\} = \partial \Omega_s \cap \partial B_r(x_s)$. We choose the labels so that $p_1,\ p_2 \in \overline{\Omega}^1$ and $p_1,\ p_4 \in \overline{\Omega}^4$. Note that by transversality we can make these choices continuously in $s$ and $r$, although we suppress the dependence on $s$ and $r$ for convenience of notation. For clarity, we refer the reader to Figure \ref{fig:setup}.

\begin{figure}
\centering
\begin{tikzpicture}
    \fill[gray!20!white] (-7, 2) .. controls (-6, 1.8) and (-6.5, 1.2) .. (-5.5, 1)
        .. controls (-4.5, .8) and (-5, .2) .. (-4, 0)
        .. controls (-5, -.2) and (-4.5, -.8) .. (-5.5, -1)
        .. controls (-6.5, -1.2) and (-6, -1.8) .. (-7, -2);
    \fill[gray!20!white] (-1, 2) .. controls (-2, 1.8) and (-1.5, 1.2) .. (-2.5, 1)
        .. controls (-3.5, .8) and (-3, .2) .. (-4, 0)
        .. controls (-3, -.2) and (-3.5, -.8) .. (-2.5, -1)
        .. controls (-1.5, -1.2) and (-2, -1.8) .. (-1, -2);
    \draw[thick] (-7, 2) .. controls (-6, 1.8) and (-6.5, 1.2) .. (-5.5, 1)
        .. controls (-4.5, .8) and (-5, .2) .. (-4, 0)
        .. controls (-3, -.2) and (-3.5, -.8) .. (-2.5, -1)
        .. controls (-1.5, -1.2) and (-2, -1.8) .. (-1, -2);
    \draw[thick] (-7, -2) .. controls (-6, -1.8) and (-6.5, -1.2) .. (-5.5, -1)
        .. controls (-4.5, -.8) and (-5, -.2) .. (-4, 0)
        .. controls (-3, .2) and (-3.5, .8) .. (-2.5, 1)
        .. controls (-1.5, 1.2) and (-2, 1.8) .. (-1, 2);
    \draw (-2.2, 0) arc [start angle = 0, end angle = 360, radius = 1.8];
    \draw (-3.2, 2) node {$B_r(x_s)$};
    \draw (-6.5, .5) node {$\Omega_s$};
    \draw (-5.2, 0) node {$\Omega^1$};
    \draw (-4, -1.2) node {$\Omega^2$};
    \draw (-2.8, 0) node {$\Omega^3$};
    \draw (-4, 1.2) node {$\Omega^4$};
    \filldraw (-4, 0) circle [radius=0.07];
    \draw (-4, -.35) node {$x_s$};
    \filldraw (-5.5, 1) circle [radius=0.07];
    \draw (-5.6, 1.3) node {$p_1$};
    \filldraw (-5.5, -1) circle [radius=0.07];
    \draw (-5.6, -1.3) node {$p_2$};
    \filldraw (-2.5, -1) circle [radius=0.07];
    \draw (-2.4, -1.3) node {$p_3$};
    \filldraw (-2.5, 1) circle [radius=0.07];
    \draw (-2.4, 1.3) node {$p_4$};
\end{tikzpicture}
\caption{Cut-and-paste setup.}
\label{fig:setup}
\end{figure}
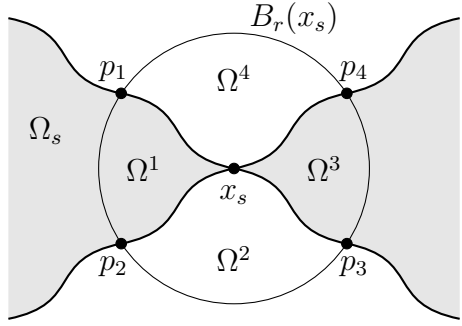

\emph{Construction of $\Omega_{s,r}^+$}: Let $\Omega^{2,+}$ be the open set satisfying
\begin{itemize}
    \item $\partial \Omega^{2,+} \cap B_r(x_s)$ is the unique length minimizing curve from $p_2$ to $p_3$ in $B_r(x_s)$,
    \item $\partial \Omega^{2,+} \setminus B_r(x_s) = \partial \Omega^2 \setminus B_r(x_s)$.
\end{itemize}
Let $\Omega^{4,+}$ be the open set satisfying
\begin{itemize}
    \item $\partial \Omega^{4,+} \cap B_r(x_s)$ is the unique length minimizing curve from $p_1$ to $p_4$ in $B_r(x_s)$,
    \item $\partial \Omega^{4,+} \setminus B_r(x_s) = \partial \Omega^4 \setminus B_r(x_s)$.
\end{itemize}
Since the points $p_i$ are disjoint and ordered, $\Omega^{2,+}$ and $\Omega^{4,+}$ are disjoint. We define
\[ \Omega_{s,r}^+ := (\Omega_s \setminus B_r(x_s)) \cup (B_r(x_s) \setminus (\Omega^{2,+} \cup \Omega^{4,+})). \]
Since $p_i$ vary continuously and we have uniqueness from Lemma \ref{lem:inj_c_radius}, $\Omega_{s,r}^+$ is continuous in $s$ and $r$.

By construction, we have
\[ \cA^c(\Omega_{s,r}^+) < \cA^c(\Omega_s) \]
for $r > 0$.

The construction of $\Omega_{s,r}^+$ is illustrated in Figure \ref{fig:omega+}.

\begin{figure}
\centering
\begin{tikzpicture}
    \fill[gray!20!white] (7, 2) .. controls (6, 1.8) and (6.5, 1.2) .. (5.5, 1)
        -- (2.5, 1)
        .. controls (1.5, 1.2) and (2, 1.8) .. (1, 2)
        -- (1, -2)
        .. controls (2, -1.8) and (1.5, -1.2) .. (2.5, -1)
        -- (5.5, -1)
        .. controls (6.5, -1.2) and (6, -1.8) .. (7, -2)
        -- (7, 2);
    \draw[thick] (7, 2) .. controls (6, 1.8) and (6.5, 1.2) .. (5.5, 1);
    \draw[thick] (2.5, -1) .. controls (1.5, -1.2) and (2, -1.8) .. (1, -2);
    \draw[thick] (7, -2) .. controls (6, -1.8) and (6.5, -1.2) .. (5.5, -1);
    \draw[thick] (2.5, 1) .. controls (1.5, 1.2) and (2, 1.8) .. (1, 2);
    \draw[thick] (2.5, 1) -- (5.5, 1);
    \draw[thick] (2.5, -1) -- (5.5, -1);
    \draw (5.8, 0) arc [start angle = 0, end angle = 360, radius = 1.8];
    \draw (4.8, 2) node {$B_r(x_s)$};
    \draw (1.6, .5) node {$\Omega_{s,r}^+$};
    \draw (4, -1.4) node {$\Omega^{2, +}$};
    \draw (4, 1.4) node {$\Omega^{4,+}$};
    \filldraw (2.5, 1) circle [radius=0.07];
    \draw (2.4, 1.3) node {$p_1$};
    \filldraw (2.5, -1) circle [radius=0.07];
    \draw (2.4, -1.3) node {$p_2$};
    \filldraw (5.5, -1) circle [radius=0.07];
    \draw (5.6, -1.3) node {$p_3$};
    \filldraw (5.5, 1) circle [radius=0.07];
    \draw (5.6, 1.3) node {$p_4$};
\end{tikzpicture}
\caption{Construction of $\Omega_{s,r}^+$.}
\label{fig:omega+}
\end{figure}
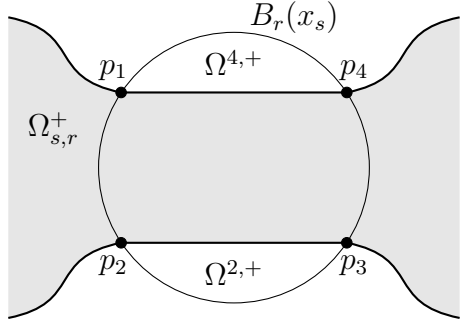

\emph{Construction of $\Omega_{s,r}^-$}: Let $\Omega^{1,-}$ be the open set satisfying
\begin{itemize}
    \item $\partial \Omega^{1,-} \cap B_r(x_s)$ is the unique $\cA^c$ minimizing curve from $p_1$ to $p_2$ in $B_r(x_s)$ with constant curvature $c$ with respect to $\Omega^{1,-}$,
    \item $\partial \Omega^{1,-} \setminus B_r(x_s) = \partial \Omega^1 \setminus B_r(x_s)$.
\end{itemize}
Let $\Omega^{3,-}$ be the open set satisfying
\begin{itemize}
    \item $\partial \Omega^{3,-} \cap B_r(x_s)$ is the unique $\cA^c$ minimizing curve from $p_3$ to $p_4$ in $B_r(x_s)$ with constant curvature $c$ with respect to $\Omega^{3,-}$,
    \item $\partial \Omega^{3,-} \setminus B_r(x_s) = \partial \Omega^3 \setminus B_r(x_s)$.
\end{itemize}
We define
\[ \Omega_{s,r}^- := (\Omega_s \setminus B_r(x_s)) \cup (\Omega^{1,-} \cup \Omega^{3,-}). \]
Since $p_i$ vary continuously and we have uniqueness from Lemma \ref{lem:inj_c_radius}, $\Omega_{s,r}^-$ is continuous in $s$ and $r$.

By Lemma \ref{lem:isoperimetric} (since we delete one copy of the overlap of $\Omega^{1,-}$ and $\Omega^{3,-}$), we have
\[ \cA^c(\Omega_{s,r}^-) < \cA^c(\Omega_s) \]
for $r > 0$.

We also note that since $B_r(x_0) \cap \Omega^1$ and $B_r(x_0) \cap \Omega^3$ are strictly $c$-convex when $s = 0$, we have that $\Omega^{1,-}$ and $\Omega^{3,-}$ are disjoint for $s = 0$.

The construction of $\Omega_{s,r}^-$ is illustrated in Figure \ref{fig:omega-}.

\begin{figure}
\centering
\begin{tikzpicture}
    \fill[gray!20!white] (-7, 2) .. controls (-6, 1.8) and (-6.5, 1.2) .. (-5.5, 1)
        .. controls (-4.75, .5) and (-4.75, -.5) .. (-5.5, -1)
        .. controls (-6.5, -1.2) and (-6, -1.8) .. (-7, -2);
    \fill[gray!20!white] (-1, 2) .. controls (-2, 1.8) and (-1.5, 1.2) .. (-2.5, 1)
        .. controls (-3.25, .5) and (-3.25, -.5) .. (-2.5, -1)
        .. controls (-1.5, -1.2) and (-2, -1.8) .. (-1, -2);
    \draw[thick] (-7, 2) .. controls (-6, 1.8) and (-6.5, 1.2) .. (-5.5, 1)
        .. controls (-4.75, .5) and (-4.75, -.5) .. (-5.5, -1)
        .. controls (-6.5, -1.2) and (-6, -1.8) .. (-7, -2);
    \draw[thick] (-1, 2) .. controls (-2, 1.8) and (-1.5, 1.2) .. (-2.5, 1)
        .. controls (-3.25, .5) and (-3.25, -.5) .. (-2.5, -1)
        .. controls (-1.5, -1.2) and (-2, -1.8) .. (-1, -2);
    \draw (-2.2, 0) arc [start angle = 0, end angle = 360, radius = 1.8];
    \draw (-3.2, 2) node {$B_r(x_s)$};
    \draw (-6.4, .5) node {$\Omega_{s,r}^-$};
    \draw (-5.35, 0) node {$\Omega^{1,-}$};
    \draw (-2.6, 0) node {$\Omega^{3,-}$};
    \filldraw (-5.5, 1) circle [radius=0.07];
    \draw (-5.6, 1.3) node {$p_1$};
    \filldraw (-5.5, -1) circle [radius=0.07];
    \draw (-5.6, -1.3) node {$p_2$};
    \filldraw (-2.5, -1) circle [radius=0.07];
    \draw (-2.4, -1.3) node {$p_3$};
    \filldraw (-2.5, 1) circle [radius=0.07];
    \draw (-2.4, 1.3) node {$p_4$};
\end{tikzpicture}
\caption{Construction of $\Omega_{s,r}^-$.}
\label{fig:omega-}
\end{figure}
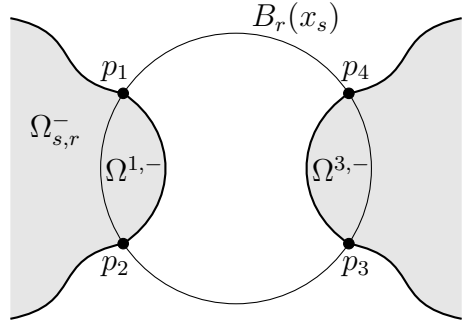

\subsection{Competitor sweepout}
Suppose there is a function $\phi$ so that the second variation of $\cA^c$ is negative. We rescale $\phi$ so that $|\phi| \leq 1$, and take $\eps > 0$, $\delta > 0$, and $r_c > 0$ as in Lemmas \ref{lem:deletion_of_overlaps} and \ref{lem:inj_c_radius}.

We construct a sweepout as follows.

\begin{enumerate}
    \item By \cite[Lemma 3.2]{KL}, there is a path from $\varnothing$ to $\Omega_{0, r_c}^-$ with $\cA^c$ below $\cA^c(\Omega_{0, r_c}^-)$.
    \item We connect $\Omega_{0, r_c}^-$ to $\Omega_{\eps, r_c}^-$ (by $\Omega_{s, r_c}^-$).
    \item We connect $\Omega_{\eps, r_c}^-$ to $\Omega_{\eps}$ (by $\Omega_{\eps, r}^-$).
    \item We connect $\Omega_{\eps}$ to $\Omega_{\eps, r_c}^+$ (by $\Omega_{\eps, r}^+$).
    \item We connect $\Omega_{\eps, r_c}^+$ to $\Omega_{0, r_c}^+$ (by $\Omega_{s, r_c}^+$).
    \item By \cite[Lemma 3.2]{KL}, there is a path from $\Omega_{0, r_c}^+$ to $M$ with $\cA^c$ below $\cA^c(\Omega_{0, r_c}^+)$.
\end{enumerate}
Since $\cA^c(\Omega_{s,r}^{\pm}) < \cA^c(\Omega)$ for $(s, r) \neq (0, 0)$ (and the value is continuous in those variables), the entire sweepout has $\cA^c$ strictly below $\cA^c(\Omega)$. Hence, $\cA^c(\Omega) > W_c$.

Figure \ref{fig:sweepout} contains an illustration of the sweepout path, and Figure \ref{fig:profile} illustrates the $\cA^c$ profile of the competitor sweepout.

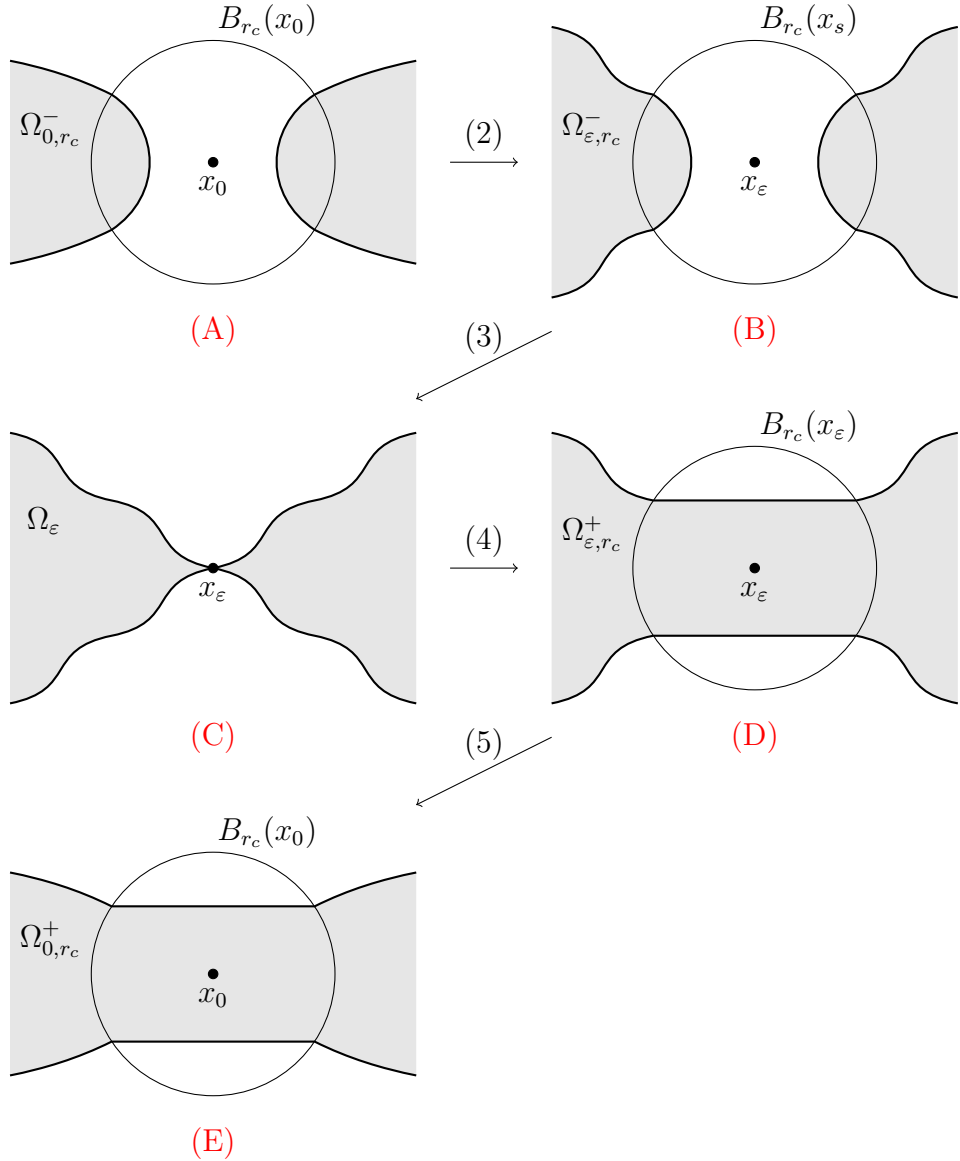
\begin{figure}
\centering
\begin{tikzpicture}[scale=.9]
    

    \fill[gray!20!white] (1, 1.5) .. controls (1.5, 1.4) and (2, 1.25) .. (2.5, 1)
        .. controls (3.25, .5) and (3.25, -.5) .. (2.5, -1)
        .. controls (2, -1.25) and (1.5, -1.4) .. (1, -1.5);
    \fill[gray!20!white] (7, 1.5) .. controls (6.5, 1.4) and (6, 1.25) .. (5.5, 1)
        .. controls (4.75, .5) and (4.75, -.5) .. (5.5, -1)
        .. controls (6, -1.25) and (6.5, -1.4) .. (7, -1.5);
    \draw[thick] (1, 1.5) .. controls (1.5, 1.4) and (2, 1.25) .. (2.5, 1)
        .. controls (3.25, .5) and (3.25, -.5) .. (2.5, -1)
        .. controls (2, -1.25) and (1.5, -1.4) .. (1, -1.5);
    \draw[thick] (7, 1.5) .. controls (6.5, 1.4) and (6, 1.25) .. (5.5, 1)
        .. controls (4.75, .5) and (4.75, -.5) .. (5.5, -1)
        .. controls (6, -1.25) and (6.5, -1.4) .. (7, -1.5);
    \draw (5.8, 0) arc [start angle = 0, end angle = 360, radius = 1.8];
    \draw (4.8, 2.1) node {$B_{r_c}(x_0)$};
    \draw (1.6, .5) node {$\Omega_{0,r_c}^-$};
    \filldraw (4, 0) circle [radius=0.07];
    \draw (4, -.35) node {$x_0$};
    \draw[color=red] (4, -2.5) node {(A)};

    \draw[->] (7.5, 0) -- (8.5, 0);
    \draw (8, 0) node[above] {(2)};

    \fill[gray!20!white] (9, 2) .. controls (10, 1.8) and (9.5, 1.2) .. (10.5, 1)
        .. controls (11.25, .5) and (11.25, -.5) .. (10.5, -1)
        .. controls (9.5, -1.2) and (10, -1.8) .. (9, -2);
    \fill[gray!20!white] (15, 2) .. controls (14, 1.8) and (14.5, 1.2) .. (13.5, 1)
        .. controls (12.75, .5) and (12.75, -.5) .. (13.5, -1)
        .. controls (14.5, -1.2) and (14, -1.8) .. (15, -2);
    \draw[thick] (9, 2) .. controls (10, 1.8) and (9.5, 1.2) .. (10.5, 1)
        .. controls (11.25, .5) and (11.25, -.5) .. (10.5, -1)
        .. controls (9.5, -1.2) and (10, -1.8) .. (9, -2);
    \draw[thick] (15, 2) .. controls (14, 1.8) and (14.5, 1.2) .. (13.5, 1)
        .. controls (12.75, .5) and (12.75, -.5) .. (13.5, -1)
        .. controls (14.5, -1.2) and (14, -1.8) .. (15, -2);
    \draw (13.8, 0) arc [start angle = 0, end angle = 360, radius = 1.8];
    \draw (12.8, 2.1) node {$B_{r_c}(x_s)$};
    \draw (9.6, .5) node {$\Omega_{\eps,r_c}^-$};
    \filldraw (12, 0) circle [radius=0.07];
    \draw (12, -.35) node {$x_{\eps}$};
    \draw[color=red] (12, -2.5) node {(B)};

    \draw[->] (9, -2.5) -- (7, -3.5);
    \draw (8, -2.6) node {(3)};

    \fill[gray!20!white] (1, -4) .. controls (2, -4.2) and (1.5, -4.8) .. (2.5, -5)
        .. controls (3.5, -5.2) and (3, -5.8) .. (4, -6)
        .. controls (3, -6.2) and (3.5, -6.8) .. (2.5, -7)
        .. controls (1.5, -7.2) and (2, -7.8) .. (1, -8);
    \fill[gray!20!white] (7, -4) .. controls (6, -4.2) and (6.5, -4.8) .. (5.5, -5)
        .. controls (4.5, -5.2) and (5, -5.8) .. (4, -6)
        .. controls (5, -6.2) and (4.5, -6.8) .. (5.5, -7)
        .. controls (6.5, -7.2) and (6, -7.8) .. (7, -8);
    \draw[thick] (7, -4) .. controls (6, -4.2) and (6.5, -4.8) .. (5.5, -5)
        .. controls (4.5, -5.2) and (5, -5.8) .. (4, -6)
        .. controls (3, -6.2) and (3.5, -6.8) .. (2.5, -7)
        .. controls (1.5, -7.2) and (2, -7.8) .. (1, -8);
    \draw[thick] (7, -8) .. controls (6, -7.8) and (6.5, -7.2) .. (5.5, -7)
        .. controls (4.5, -6.8) and (5, -6.2) .. (4, -6)
        .. controls (3, -5.8) and (3.5, -5.2) .. (2.5, -5)
        .. controls (1.5, -4.8) and (2, -4.2) .. (1, -4);
    \draw (1.5, -5.3) node {$\Omega_{\eps}$};
    \filldraw (4, -6) circle [radius=0.07];
    \draw (4, -6.35) node {$x_{\eps}$};
    \draw[color=red] (4, -8.5) node {(C)};

    \draw[->] (7.5, -6) -- (8.5, -6);
    \draw (8, -6) node[above] {(4)};

    \fill[gray!20!white] (15, -4) .. controls (14, -4.2) and (14.5, -4.8) .. (13.5, -5)
        -- (10.5, -5)
        .. controls (9.5, -4.8) and (10, -4.2) .. (9, -4)
        -- (9, -8)
        .. controls (10, -7.8) and (9.5, -7.2) .. (10.5, -7)
        -- (13.5, -7)
        .. controls (14.5, -7.2) and (14, -7.8) .. (15, -8)
        -- (15, -4);
    \draw[thick] (15, -4) .. controls (14, -4.2) and (14.5, -4.8) .. (13.5, -5);
    \draw[thick] (10.5, -7) .. controls (9.5, -7.2) and (10, -7.8) .. (9, -8);
    \draw[thick] (15, -8) .. controls (14, -7.8) and (14.5, -7.2) .. (13.5, -7);
    \draw[thick] (10.5, -5) .. controls (9.5, -4.8) and (10, -4.2) .. (9, -4);
    \draw[thick] (10.5, -5) -- (13.5, -5);
    \draw[thick] (10.5, -7) -- (13.5, -7);
    \draw (13.8, -6) arc [start angle = 0, end angle = 360, radius = 1.8];
    \draw (12.8, -3.9) node {$B_{r_c}(x_{\eps})$};
    \draw (9.6, -5.5) node {$\Omega_{\eps, r_c}^+$};
    \filldraw (12, -6) circle [radius=0.07];
    \draw (12, -6.35) node {$x_{\eps}$};
    \draw[color=red] (12, -8.5) node {(D)};

    \draw[->] (9, -8.5) -- (7, -9.5);
    \draw (8, -8.6) node {(5)};

    \fill[gray!20!white] (1, -10.5) .. controls (1.5, -10.6) and (2, -10.75) .. (2.5, -11)
        -- (5.5, -11)
        .. controls (6, -10.75) and (6.5, -10.6) .. (7, -10.5)
        -- (7, -13.5)
        .. controls (6.5, -13.4) and (6, -13.25) .. (5.5, -13)
        -- (2.5, -13)
        .. controls (2, -13.25) and (1.5, -13.4) .. (1, -13.5);
    \draw[thick] (1, -10.5) .. controls (1.5, -10.6) and (2, -10.75) .. (2.5, -11)
        -- (5.5, -11)
        .. controls (6, -10.75) and (6.5, -10.6) .. (7, -10.5);
    \draw[thick] (1, -13.5) .. controls (1.5, -13.4) and (2, -13.25) .. (2.5, -13)
        -- (5.5, -13)
        .. controls (6, -13.25) and (6.5, -13.4) .. (7, -13.5);
    \draw (5.8, -12) arc [start angle = 0, end angle = 360, radius = 1.8];
    \draw (4.8, -9.9) node {$B_{r_c}(x_0)$};
    \draw (1.6, -11.5) node {$\Omega_{0,r_c}^+$};
    \filldraw (4, -12) circle [radius=0.07];
    \draw (4, -12.35) node {$x_0$};
    \draw[color=red] (4, -14.5) node {(E)};


\end{tikzpicture}
\caption{Competitor sweepout from \S\ref{sec:cutandpaste}.}
\label{fig:sweepout}
\end{figure}

\begin{figure}
\centering
\begin{tikzpicture}
    \draw[->] (0, 3) -- (14, 3) node[right] {$t$};
    \draw[->] (0, 1.5) -- (0, 7.5) node[above] {$\cA^c$};
    \draw (-.3, 3) node[left] {$0$};

    \draw (0, 7) node[left] {$\cA^c(\Omega)$};
    \draw[dashed] (0, 7) -- (14, 7);

    \draw (0, 2) node[left] {$\cA^c(M)$};
    \draw[dashed] (0, 2) -- (14, 2);

    \draw (0, 3) .. controls (.5, 5) and (1, 5.5) .. (2, 5.5);
    \draw (2, 5.5) .. controls (2.5, 8) and (3.5, 3) .. (4, 5);
    \draw (4, 5) .. controls (4.5, 3.5) and (5.5, 5.5) .. (6, 6);
    \draw (6, 6) .. controls (6.5, 5.5) and (7.5, 2) .. (8, 4);
    \draw (8, 4) .. controls (8.5, 1) and (9.5, 8) .. (10, 5);
    \draw (10, 5) .. controls (11, 5) and (11.5, 4) .. (12, 2);

    \draw[thick] (1, 4.6) node {(1)};
    \draw[thick] (3, 6.1) node {(2)};
    \draw[thick] (5, 5.3) node {(3)};
    \draw[thick] (7, 4.9) node {(4)};
    \draw[thick] (9, 5.5) node {(5)};
    \draw[thick] (11, 3.9) node {(6)};

    \filldraw[color=red] (0, 3) circle [radius=0.1];
    \filldraw[color=red] (2, 5.5) circle [radius=0.1];
    \draw[color=red] (2, 5.5) node[below] {(A)};
    \filldraw[color=red] (4, 5) circle [radius=0.1];
    \draw[color=red] (4, 5) node[above] {(B)};
    \filldraw[color=red] (6, 6) circle [radius=0.1];
    \draw[color=red] (6, 6) node[above] {(C)};
    \filldraw[color=red] (8, 4) circle [radius=0.1];
    \draw[color=red] (8, 4) node[above] {(D)};
    \filldraw[color=red] (10, 5) circle [radius=0.1];
    \draw[color=red] (10, 5) node[below] {(E)};
    \filldraw[color=red] (12, 2) circle [radius=0.1];
\end{tikzpicture}
\caption{$\cA^c$ profile of the competitor sweepout in \S\ref{sec:cutandpaste}.}
\label{fig:profile}
\end{figure}
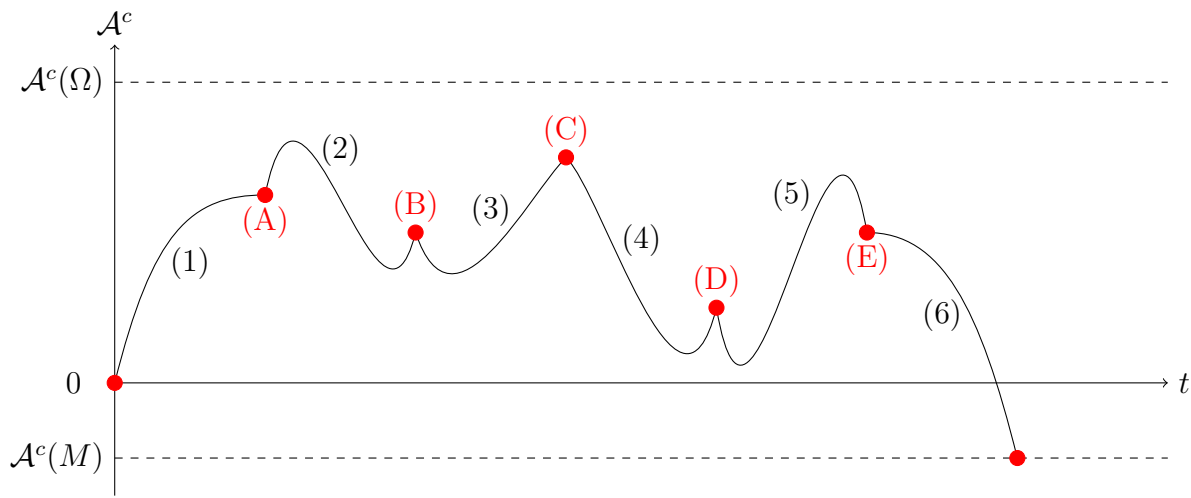

\section{Applying the Cut-and-Paste Sweepout}\label{sec:cutandpaste_apply}
In this section, we apply Theorem \ref{thm:cut_and_paste_sweepout} in each of the two configurations. In the subsequent sections, we use different techniques to obtain an improvement from Theorem \ref{thm:cut_and_paste_sweepout} in the case of Configuration 2.

\subsection{Excluding Configuration 1} We take $\phi \equiv 1$ in (\ref{eqn:second_variation_config1}), which yields
\begin{align*}
    \frac{d^2}{ds^2}\Big|_{s=0^+} (\mathrm{length}(\gamma_s) - c\M(T_s))
    & \leq -\mathrm{length}(\gamma)(\min K_g + c^2) - \frac{4c}{\sin\alpha}(1 + \cos\alpha)\\
    & = -\mathrm{length}(\gamma)(\min K_g + c^2) - 4c\cot(\alpha/2)\\
    & < -\mathrm{length}(\gamma)(\min K_g + c^2).
\end{align*}
By Theorem \ref{thm:cut_and_paste_sweepout}, Configuration 1 cannot achieve the first $c$ min-max width if
\begin{equation}\label{eqn:config1}
    c^2 \geq -\min K_g.
\end{equation}
Since $\coth(x) \geq 1$ for $x > 0$, we see that (\ref{eqn:formula_general_surface}) implies (\ref{eqn:config1}).

\subsection{Excluding Configuration 2} \label{ssec:exclude_c2}Let $L = \mathrm{length}(\gamma^+)$.

We view $\Omega$ intrinsically as a surface with strictly convex boundary. Let $p$ and $q$ be the points in the boundary of $\Omega$ corresponding to the self intersection point $x_0$. Let $\sigma$ be a length-minimizing curve in $\Omega$ from $p$ to $q$. Then $\sigma$ is a geodesic loop in $M$ (based at $x_0$). By the definition of the injectivity radius, we have $\mathrm{length}(\sigma) \geq 2\,\mathrm{inj}(M,g)$. Since $\sigma$ is length minimizing, we have
\begin{equation}\label{eqn:config2_length_inj}
    L > 2\,\mathrm{inj}(M, g).
\end{equation}

Let $\phi(t) = \sin(\pi t/L)$, where $[0, L]$ parametrizes $\gamma^+$. Then (\ref{eqn:second_variation_config2}) gives
\begin{align*}
    \frac{d^2}{ds^2}\Big|_{s=0^+} (\mathrm{length}(\gamma_s) - c\M(T_s))
    & = \left(\frac{\pi^2}{L^2} -(\min K_g + c^2)\right)\frac{L}{2}.
\end{align*}
By Theorem \ref{thm:cut_and_paste_sweepout}, Configuration 2 cannot achieve the first $c$ min-max width if
\begin{equation}\label{eqn:config2}
    c^2 \geq \frac{\pi^2}{4\,\mathrm{inj}(M,g)^2} - \min K_g.
\end{equation}
We observe that the constant scaling $\mathrm{inj}(M,g)^{-1}$ in this formula is $\pi/2$, the same nonsharp constant in the corresponding formula from \cite{SchneiderS2}.

\section{Comparison Argument}\label{sec:comparison}
We now conclude the proof of Theorem \ref{thm:general_surface} by using a comparison argument based on a result of \cite{dekster} in order to rule out Configuration 2 when condition \eqref{eqn:formula_general_surface} is satisfied.

The result we shall need, which follows from the Corollary to Theorem 1 in \cite{dekster}, is as follows:

\begin{proposition}[\cite{dekster}]\label{prop:dekster_comparison}
    Let $(M^2,g)$ be a smooth closed oriented Riemannian surface and let $\Omega\subset M$ be an open connected region with compact closure, bounded by a nonempty smooth curve. Suppose that:
    \begin{itemize}
        \item $\partial \Omega$ is $c$-convex to $\Omega$;
        \item $K_g\geq k>-c^2$ in $\Omega$.
    \end{itemize}
    Then, if $\sigma$ is a geodesic of $M$ which lies in $\Omega$,
    \begin{equation}
        \mathrm{length}(\sigma)\leq 2 R_0(c,k),
    \end{equation}
    where
    \[ R_0(c,k) := \begin{cases}
    \frac{1}{\sqrt{k}}\cot^{-1}\left(\frac{c}{\sqrt{k}}\right) & k > 0\\
    1/c & k = 0\\
    \frac{1}{\sqrt{-k}}\coth^{-1}\left(\frac{c}{\sqrt{-k}}\right) & k < 0.
\end{cases} \]
\end{proposition}

As in \S\ref{ssec:exclude_c2}, we note that, in the case of Configuration 2, $\Omega$ is a connected region in $M$, with $c$-convex boundary $\partial\Omega=\gamma(S^1)$. Let $p$ and $q$ be the points in the boundary of $\Omega$ corresponding to the self-intersection point $x_0$, and $\sigma$ be a length-minimizing curve in $\Omega$ from $p$ to $q$. Then $\sigma$ is a geodesic loop in $M$ of length $L \geq 2\,\mathrm{inj}(M, g)$.

Note that we can easily approximate $\Omega$ by connected regions with smooth $c$-convex boundary. For example, we can `round off the corner' at $x_0$ by smooth arcs contained in $\Omega\cap B_r(x_0)$ with curvature at least $c$ for $0<r\leq r_c$. The resulting set $\tilde{\Omega}_r$ is a connected region with smooth $c$-convex boundary, and $\sigma_r\coloneqq\sigma\cap\tilde{\Omega}_r$ is a geodesic of $M$ which lies in $\tilde{\Omega}_r$. Hence, by Proposition \ref{prop:dekster_comparison}
\begin{equation}
    \mathrm{length}(\sigma_r)\leq 2 R_0(c,\min K_g).
\end{equation}

However, as it is clear by the construction of $\tilde{\Omega}_r$, $\mathrm{length}(\sigma_r)\to\mathrm{length}(\sigma)$ as $r\to 0$.
Hence, 
\begin{equation}\label{eqn:geometric_constraint}
    2\,\mathrm{inj}(M,g)\leq\mathrm{length}(\sigma)\leq 2 R_0(c,\min K_g),
\end{equation}
which completes the proof of Theorem \ref{thm:general_surface}.

\section{Positive Curvature}\label{sec:positive_curvature}

We shall now restrict to the case of positive ambient curvature. Throughout this section, we let $M=S^2$ and assume without loss of generality that $\max K_g=1$.

In order to prove Theorem \ref{thm:positive_curvature}, we produce a competitor sweepout $\{D_t\}_{t\in [0,1]}$ to show Configuration 2 cannot achieve the first min-max width under the assumptions of the theorem.

The outline of the construction is as follows.

\subsection{Outline of the construction}\label{ssec:outline} Suppose we are in the case of Configuration 2, so $\Omega$ is a connected region in $S^2$, with a $c$-convex boundary $\partial\Omega=\gamma(S^1)$, and let $p$ and $q$ be the points in the boundary of $\Omega$ corresponding to the self-intersection point $x_0$. Let us represent $S^1$ as the closed interval $[0,2\pi]$ with endpoints identified and let us suppose $x_0=\gamma(0)=\gamma(\theta_0)=\gamma(2\pi)$ where $\theta_0$ lies in the open interval $(0,2\pi)$.
We can then write:
\begin{equation*}
    \partial \Omega=\gamma^1\cup\gamma^2
\end{equation*}
where both $\gamma^1\coloneqq\gamma([0,\theta_0])$ and $\gamma^2\coloneqq\gamma([\theta_0,2\pi])$ are embedded loops in $S^2$, based at $x_0$.

Before we describe the construction of the sweepout, let us also note that by the classical Jordan-Schoenflies Theorem, any Jordan curve $\Gamma$ in $S^2$ separates $S^2$ into the union of two topological disks, $D$ and $S^2\setminus D$, having $\Gamma$ as their common boundary. 
Finally, if $\{\Gamma_t\}$ is a continuous family of Jordan curves in $S^2$ with uniformly bounded length, we can select a component $D_t$ of $S^2\setminus\Gamma_t$ in such a way that $t\mapsto \mathbf{1}_{D_t}$ is continuous in $L^1$. 

In order to construct the competitor sweepout, we start with the following continuous paths of embedded loops in $S^2$ with controlled lengths.  The corresponding disks $\{D_t\}$ will provide the desired sweepout of $S^2$.
\begin{enumerate}
    \item We connect a constant loop (with corresponding $D_0=\emptyset$) to $\gamma^1$ by $\{\gamma^1_t\}$.
    \item We connect $\gamma^1$ to $\gamma^2$ by $\{\sigma_t\}$.
    \item We connect $\gamma^2$ to a constant loop (with corresponding $D_1=S^2$) by $\{\gamma^2_t\}$. 
\end{enumerate}

We shall show that along path (3), $\cA^c$ stays below its initial value $\cA^c(D^2)$, where $D^2$ is the component of $S^2\setminus\gamma^2$ which contains $\Omega$. Note that $$\cA^c(D^2)<\cA^c(\Omega).$$
As for paths (1) and (2), observe that for all $t$ we have $\cA^c(D_t)\leq\mathrm{length}(\partial D_t)$.
Hence,
\begin{equation}\label{eqn:sup_A^c}
    \sup_t\cA^c(D_t) \leq \max\{\sup_t\mathrm{length}(\gamma^1_t),\sup_t\mathrm{length}(\sigma_t), \length(\gamma^2)\}.
\end{equation}
    
We now describe the three paths in detail.

\begin{figure}
\centering
\begin{tikzpicture}
    \fill[gray!10!white] (-5.5, .97) arc [start angle=75.5, end angle=-75.5, x radius=2, y radius = 1]
        .. controls (-5.6, -.9) and (-5.9, -.5) .. (-6, 0)
        .. controls (-5.9, .5) and (-5.6, .9) .. (-5.5, .97);
    \fill[gray!25!white] (-5.5, .97) arc [start angle=75.5, end angle=-75.5, x radius=2, y radius = 1]
        .. controls (-5.2, -.5) and (-5.2, .5) .. (-5.5, .97);
    \draw (-6, 0) ellipse (2 and 1);
    \filldraw (-6, 0) circle [radius=.02];
    \draw[dashed] (-6, 0) .. controls (-5.9, .5) and (-5.6, .9) .. (-5.5, .97);
    \draw[dashed] (-6, 0) .. controls (-5.9, -.5) and (-5.6, -.9) .. (-5.5, -.97);
    \draw (-5.5, .97) .. controls (-5.2, .5) and (-5.2, -.5) .. (-5.5, -.97);
    \draw (-6, -1.5) node {$\gamma^1$};

    \draw[->] (-3.5, 0) -- (-2.5, 0);

    \fill[gray!25!white] (0, 1) arc [start angle=90, end angle=-90, x radius=2, y radius = 1];
    \draw (0, 0) ellipse (2 and 1);
    \draw (0, 1) -- (0, -1);
    \draw (0, -1.5) node {$\sigma_t$};

    \draw[->] (2.5, 0) -- (3.5, 0);

    \fill[gray!15!white] (5.5, .97) arc [start angle=180-75.5, end angle=-180+75.5, x radius=2, y radius = 1]
        .. controls (5.2, -.5) and (5.2, .5) .. (5.5, .97);
    \fill[gray!25!white] (5.5, .97) arc [start angle=180-75.5, end angle=-180+75.5, x radius=2, y radius = 1]
        .. controls (5.6, -.9) and (5.9, -.5) .. (6, 0)
        .. controls (5.9, .5) and (5.6, .9) .. (5.5, .97);
    \draw (6, 0) ellipse (2 and 1);
    \filldraw[color=gray] (6, 0) circle [radius=.02];
    \draw[dashed, color=gray] (6, 0) .. controls (5.9, .5) and (5.6, .9) .. (5.5, .97);
    \draw[dashed, color=gray] (6, 0) .. controls (5.9, -.5) and (5.6, -.9) .. (5.5, -.97);
    \draw (5.5, .97) .. controls (5.2, .5) and (5.2, -.5) .. (5.5, -.97);
    \draw (6, -1.5) node {$\gamma^2$};
\end{tikzpicture}
\caption{An illustration of the length min-max sweepout $\sigma_t$.}
\label{fig:length_sweepout}
\end{figure}
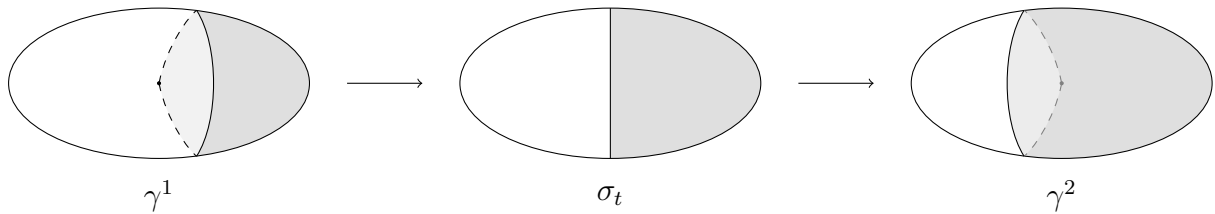

\subsection{Path 1} In order to construct path 1, we use curve shortening flow.

Let $\sigma$ be a length-minimizing curve in $\Omega$ from $p$ to $q$, so that $\sigma$ is an embedded geodesic loop in $S^2$, based at $x_0$. By a previous observation, $\sigma$ separates $S^2$ into the union of two topological disks having $\sigma$ as their common boundary and $\sigma$ is weakly convex to (at least) one of them\footnote{$\sigma$ is weakly convex to both disks if and only if $\sigma$ is a closed geodesic of $(S^2,g)$.}. Let this disk be $\Delta^1$. Also, note that exactly one of $\gamma^1$ and $\gamma^2$ lies entirely in the closure of $\Delta^1$.

By reparametrizing $\gamma$ in order to swap the curves $\gamma^1$ and $\gamma^2$ if needed, let $\gamma^1\subset\overline{\Delta^1}$.

We want to run curve shortening flow with $\gamma^1$ as initial condition.

We first approximate $\gamma^1$ by a smooth embedded Jordan curve $\tilde{\gamma}^1$ so that
\begin{itemize}
    \item $\tilde{\gamma}^1$ lies in the interior of $\Delta^1$, 
    \item $\length(\tilde{\gamma}^1)<\length(\gamma^1)$,
    \item $\tilde{\gamma}^1$ is homotopic to $\gamma^1$ through smooth embedded Jordan curves lying in $\Delta^1$, with length bounded above by the length of $\gamma^1$.
\end{itemize}
We can easily do this by `rounding off the corner' at $x_0$. 

Let $\tilde{D}^1$ be the unique component of $S^2\setminus\tilde{\gamma}^1$ which lies in the interior of $\Delta^1$.

Let us now run curve shortening flow with the smooth embedded closed curve $\tilde{\gamma}^1$ as initial condition. By the the work of Gage-Hamilton, Gage and Grayson (\cite{gagehamilton}, \cite{gage} and \cite{grayson}), this produces a maximal flow $\{\Gamma^1_t\}_{t\in[0,T_1)}$ , with $\length(\Gamma_t^1
)\leq\length(\tilde{\gamma}^1)<\length(\gamma^1)$ for all $t\in[0,T_1)$ and each $\Gamma^1_t$ bounds a disk $\tilde{D}^1_t$ so that
\begin{itemize}
    \item $t\to \Gamma_t^1$ is continuous in the flat topology,
    \item $t\to \mathbf{1}_{\tilde{D}^1_t}$ is continuous in $L^1$;
    \item $\tilde{D}^1_0=\tilde{D}^1$, as defined above.
\end{itemize}

By the maximum principle, as $\sigma$ is convex to $\Delta^1$, the curve $\Gamma_t$ lies in $\Delta^1$ for all $t\in[0,T_1)$. However, as $\min K_g>0$, there are no closed geodesics of $(S^2,g)$ entirely contained in the interior of $\Delta^1$. Hence, by \cite[Theorem 0.1]{grayson}, $T_1<+\infty$ and $\Gamma^1_t$ converges to a point curve in $\Delta^1$, and $\tilde{D}^1_t\to\emptyset$.

By reversing and rescaling time, this produces the desired path between the empty loop and $\tilde{\gamma}^1$. Finally, by gluing this path together with the approximating homotopy, we have constructed a continuous path of embedded loops in $S^2$, joining a constant loop to $\gamma^1$, with length bounded above by $\length(\gamma^1)$. 

\subsection{Path 3} We observe that there is a unique component $D^2$ of $S^2\setminus \gamma^2$ which contains $\Omega$. As noted above, this component is a topological disk with boundary $\gamma^2$. 

Then, note that $\gamma^2$ is $c$-concave (see \cite[page 3]{KL}) to the disk $D^2$. By modifying  $\gamma^2$ and the set $D^2$ as in the construction of $\Omega^{1,-}$ in \ref{ssec:cut-paste}, we obtain an open disk $D^{2,-}$ with piecewise smooth $c$-concave boundary, and a point on the boundary with vanishing geodesic curvature. Then, \cite[Lemma 3.2]{KL} provides a continuous family of sets interpolating between $D^{2,-}$ and $S^2$ with $\cA^c$ bounded above by its initial value $\cA^c(D^{2,-})<\cA^c(D^2)$.

\subsection{Path 2} Finally, we want to connect $\gamma^1$ to $\gamma^2$ by a continuous path $\{\sigma_t\}$ of embedded loops in $\Omega$ with controlled length. 

As in \ref{ssec:exclude_c2}, we view $\Omega$ intrinsically as a surface with strictly convex boundary, with $p$ and $q$ being the points in the boundary of $\Omega$ corresponding to the self intersection point $x_0$. We consider sweepouts of $\Omega$ by curves connecting $p$ to $q$ in $\Omega$, starting at $\gamma^1$ and ending at $\gamma^2$.

Let $\mathcal{P}$ be the space of smooth embedded curves in $\overline{\Omega}$ from $p$ to $q$, equipped with the $W^{1,2}$ topology (see Appendix \ref{app:min-max}). Let $\mathcal{E}$ be the space of continuous paths in $\mathcal{P}$ from $\gamma^1$ to $\gamma^2$. We define
\[ w_0(\Omega) := \inf_{\Gamma \in \mathcal{E}} \max_t \mathrm{length}(\Gamma_t). \]




We shall prove the following result in Appendix (\ref{app:min-max}).

\begin{proposition}\label{prop:min-max_Omega}
There is a sequence $\{\Gamma^j\}\subset\mathcal{E}$ such that 
\[\max_{t\in [0,1]}\length(\Gamma^j_t)\to w_0(\Omega)\]
as $j\to\infty$. Moreover, if
\[w_0(\Omega)>\max\{\length(\gamma^1), \length(\gamma^2)\},\]
then there is an embedded geodesic $\tilde{\sigma}$ in $\Omega$ joining $p$ to $q$, such that
    \begin{equation}
        \length(\tilde{\sigma})=w_0(\Omega).
    \end{equation}
\end{proposition}

By Proposition \ref{prop:min-max_Omega}, for each positive integer $j$ there is a continuous path $\{\sigma^j_t\}$ of embedded loops based at $x_0$ connecting $\gamma^1$ to $\gamma^2$ so that
\[\max_{t\in[0,1]}\length(\sigma^j_t)\leq w_0(\Omega)+\frac{1}{j}.\]
For each $j$, there is a corresponding continuous path $\{D^j_t\}$ of open disks with $\partial D^j_t=\sigma^j_t$ where $D^j_t$ is chosen to be the unique component of $S^2\setminus\sigma^j_t$ such that $\gamma^1\subset \overline{D^j_t}$.

In particular, 
\[\cA^c(D^j_t)=\length(\sigma^j_t)-c\cH^2(D^j_t)<\length(\sigma^j_t)\leq w_0(\Omega)+\frac{1}{j}.\]

By concatenation with paths 1 and 3, for each positive integer $j$ we obtain a competitor sweepout $\{D^j_t\}$ such that
\begin{align*}
\begin{split}
    \sup_t\cA^c(D^j_t) &\leq \max\{\sup_t\mathrm{length}(\gamma^1_t),\sup_t\mathrm{length}(\sigma_t), \length(\gamma^2)\}\\
    &\leq \max\{\mathrm{length}(\gamma^1),w_0(\Omega)+\frac{1}{j}, \length(\gamma^2)\}.
\end{split}
\end{align*}
since $\length(\gamma^1_t)\leq\length(\gamma^1)$ by construction.

Therefore,
\begin{equation}\label{eqn:seq_competitors}
    W_c\leq\inf_j\sup_t\cA^c(D^j_t)\leq \max\{\mathrm{length}(\gamma^1),w_0(\Omega), \length(\gamma^2)\}.
\end{equation}

Finally, if $w_0(\Omega)>\max\{\length(\gamma^1)\length(\gamma^2)\}$, then by Proposition \ref{prop:min-max_Omega}, $w_0(\Omega)$ is the length of an embedded geodesic loop based at $x_0$. This will be crucial in the next subsection.

\subsection{Estimates for the first \textit{c}-min-max width} We use the competitor sweepout $\{D_t\}$ to provide estimates for the first $c$-min-max width $W_c$ of $(S^2,g)$.

Assume $\Omega$ achieves the min-max width $W_c$, so that 
\begin{align}\label{eqn:width}
\begin{split}
    W_c = \length(\gamma^1)+\length(\gamma^2)-c\cH^2(\Omega).
\end{split}
\end{align}

\subsubsection{Lower bound} By \eqref{eqn:config2_length_inj}, for each $i=1,2$ 
\[\length(\gamma^i)>2\,\mathrm{inj}(S^2,g).\]
Moreover, since $0<\min K_g\leq K_g\leq 1$, by a classical result of Klingenberg (see, e.g. \cite[Theorem 6.5.1]{petersen}), $\mathrm{inj}(S^2,g)\geq \pi$, so 
\begin{equation}\label{eqn:length_lowerbound}
    \length(\gamma^i) > 2\pi.
\end{equation}

On the other hand, we can also apply the Gauss-Bonnet Theorem to the region $\Omega$:
\begin{equation}\label{eqn:GB}
    (\min K_g)\,\cH^2(\Omega)\leq\int_\Omega K_g\ d\cH^2=2\alpha-c\,\length(\gamma)\leq 2\pi-4\pi c
\end{equation}
and obtain 
\begin{equation}\label{eqn:area_upperbound}
    \cH^2(\Omega)\leq\frac{2\pi(1-2c)}{\min K_g}.
\end{equation}
Hence, by \eqref{eqn:width}, \eqref{eqn:length_lowerbound} and \eqref{eqn:area_upperbound}, we get
\begin{equation}\label{eqn:width_lowerbound}
    W_c> 4\pi-\frac{2\pi c(1-2c)}{\min K_g}
\end{equation}

\subsubsection{Upper bound} By \eqref{eqn:seq_competitors},
\begin{equation*}
    W_c\leq \max\{\mathrm{length}(\gamma^1),w_0(\Omega), \length(\gamma^2)\}.
\end{equation*}

Assume without loss of generality $\max\{\length(\gamma^1), \length(\gamma^2)\}=\length(\gamma^1)$.

We have two cases to analyze.
\begin{itemize}
    \item If $w_0(\Omega)=\length(\gamma^1)$, then $W_c\leq \length(\gamma^1)$, so by \eqref{eqn:width},
    \[\length(\gamma^1)+\length(\gamma^2)-c\cH^2(\Omega)\leq\length(\gamma^1).\]
    So, combining this with \eqref{eqn:length_lowerbound} and \eqref{eqn:area_upperbound}, we get
    \begin{equation}\label{eqn:condition1}
        2\pi<\length(\gamma^2)\leq c\cH^2(\Omega)\leq\frac{2\pi c(1-2c)}{\min K_g}.
    \end{equation}
    \item On the other hand, if $w_0(\Omega)>\length(\gamma^1)$, then by Proposition \ref{prop:min-max_Omega}, the width is achieved by the length of a geodesic $\tilde{\sigma}$ from $p$ to $q$. By Proposition \ref{prop:dekster_comparison}, and the approximation argument in Section \ref{sec:comparison}, $\length(\tilde{\sigma})\leq 2R_0(c,\min K_g)$. Hence,
    \begin{equation}
        W_c\leq w_0(\Omega)\leq 2 R_0(c,\min K_g)=\frac{2}{\sqrt{\min K_g}}\cot^{-1}\left(\frac{c}{\sqrt{\min K_g}}\right).
    \end{equation}
    By combining this with \eqref{eqn:width_lowerbound}, we get 
    \begin{equation}\label{eqn:condition2}
        2\pi-\frac{\pi c(1-2c)}{\min K_g}< \frac{1}{\sqrt{\min K_g}}\cot^{-1}\left(\frac{c}{\sqrt{\min K_g}}\right).
    \end{equation}
\end{itemize}

Finally, we can combine the constraints given by \eqref{eqn:condition1} and \eqref{eqn:condition2} with the geometric constraint \eqref{eqn:geometric_constraint}, which in positive ambient curvature gives
\begin{equation}\label{eqn:geometric_constraint2}
    c\leq\sqrt{\min K_g}\cot(\pi\sqrt{\min K_g}).
\end{equation}

In summary, if either of the following two conditions hold:
\begin{equation}\label{eqn:conditions}
    \begin{cases}
        c > \sqrt{\min K_g}\cot(\pi\sqrt{\min K_g})\\
        2\pi \geq\max\left\{\frac{2\pi c(1-2c)}{\min K_g},\frac{\pi c(1-2c)}{\min K_g}+\frac{1}{\sqrt{\min K_g}}\cot^{-1}\left(\frac{c}{\sqrt{\min K_g}}\right) \right\},
    \end{cases}
\end{equation}
then we reach a contradiction to the assumption that $\Omega$ achieves the first $c$-min-max width. Figure \ref{fig:comparison_pos} illustrates the region defined by (\ref{eqn:conditions}). We emphasize that this region contains all positive values of $c$ when $\min K_g \geq .1167$, and it contains all positive values of $\min K_g$ when $c \geq 1/\pi$.

\section{Index Bounds}\label{sec:index}
The Morse index upper bound follows from the general deformation arguments of \cite{MNindex} (see also the modification to the $\cA^c$ setting in \cite{zhou}). In particular, since we directly prove regularity, we can avoid the restriction to ambient dimension at least 3, which is only assumed in those papers to ensure regularity.

Since (\ref{eqn:formula_general_surface}) implies $\min K_g + c^2 > 0$, it follows from taking a constant test function in the second variation formula that any smooth closed embedded curve of constant curvature $c$ satisfying the hypotheses of Theorem \ref{thm:general_surface} or Theorem \ref{thm:positive_curvature} is not stable. Hence, the $\cA^c$ Morse index is 1.

\appendix
\section{Min-max for length in surfaces with convex boundary}\label{app:min-max}

In this appendix, we prove Proposition \ref{prop:min-max_Omega}. Let the setting be as in Section \ref{sec:positive_curvature}.  We can also assume without loss of generality that $\Omega$ is isometrically embedded in some Euclidean space $\R^N$. Let $p$ and $q$ be the points in the boundary of $\Omega$ corresponding to the self intersection point $x_0$. 

Let $I$ be the closed unit interval $[0,1]$, and let $X$ be the space 
\[X=\{\sigma:I\to\Omega\mid \sigma\in W^{1,2}(I,\Omega),\sigma(0)=p, \sigma(1)=q\}\footnote{Note that by the Sobolev embedding theorem, $W^{1,2}(I,\Omega)$ embeds continuously into $ C^{0,\frac{1}{2}}(I,\Omega)$. Thus, we can assume without loss of generality that elements of $X$ are continuous functions which satisfy the prescribed boundary conditions.}\subset W^{1,2}(I,\R^N)\]
endowed with the $W^{1,2}$ norm.
Note that for a map $\sigma\in X$ we can define its length and Dirichlet energy by:
\[\length(\sigma)=\int_0^1|\sigma'(\tau)|\ d\tau\]
and
\[\mathrm{E}(\sigma)\coloneqq\int_0^1|\sigma'(\tau)|^2\ d\tau.\]
H\"{o}lder's inequality implies that for all $\sigma\in X$, $\length(\sigma)^2\leq\mathrm{E}(\sigma)$, with equality if and only if $|\sigma'|=\length(\sigma)$ almost everywhere.

Let us also denote by $X_K$ the subset of $X$ consisting of maps with energy bounded above by $K^2$ (and, hence, length bounded above by $K$).

\subsection{The (fixed-endpoint) curve shortening map}

Let $L$ be a large positive integer to be specified later.
As in Colding-Minicozzi \cite{CM}, we say a map $\sigma:I\to\Omega$ is \textit{piecewise-linear with $L$ breaks} if there is a partition $0=\tau_0<\tau_1<\dots<\tau_{L}<\tau_{L+1}=1$ of $I$ such that for each $i\in\{0, \dots, L\}$, $\sigma\vert_{[\tau_{i},\tau_{i+1}]}$ is a constant-speed geodesic segment in $\Omega$ (i.e. in this terminology, a \textit{linear map}).

For $r_0\in(0,\mathrm{inj}(S^2,g))$, we define $\Lambda_{L,r_0}\subset X$ to be the space of piecewise linear maps with $L$ breaks, such that the length of each geodesic segment is at most $r_0$, and endow it with the $W^{1,2}$ topology. Finally, let $G\subset\Lambda_{L,r_0}$ be the set of immersed geodesics in $\Omega$ between $p$ and $q$, with length at most $(L+1)r_0$.

We shall from now on denote $\Lambda_{L,r_0}$ simply by $\Lambda$.

As in Colding-Minicozzi \cite{CM} (see also \cite{xinFB}), we now define a (fixed-endpoint) curve shortening map $\Psi:X_{(L+1)r_0}\to\Lambda$, which mimics the construction of Birkhoff's curve shortening process.

\subsubsection{Construction of the (fixed-endpoint) curve shortening map} Our construction of the fixed-endpoint curve shortening map $\Psi$ is essentially analogous to the constructions in \cite[Chapter 5, Section 2.2]{CM} and \cite[3.1]{xinFB}. We briefly summarize the main steps.

Let $\sigma\in X_{(L+1)r_0}$.

\begin{itemize}
\setlength{\itemindent}{1em}
    \item[\textbf{Step 0.}] Partition $I=[0,1]$ into $2L+2$ intervals of equal length, by choosing break points $\tau_i=\frac{i}{2L+2}$.
    \item[\textbf{Step 1.}] For $j=0,1,\dots, L$, replace $\sigma\vert_{[\tau_{2j},\tau_{2j+2}]}$ by the unique minimizing geodesic in $\Omega$ between $\sigma(\tau_{2j})$ and $\sigma(\tau_{2j+2})$, to obtain a piecewise linear map $\sigma_e:I\to\Omega$.
    \item[\textbf{Step 2.}] Reparametrize $\sigma_e$ to get a constant-speed curve $\tilde{\sigma}_e$ with the same image.
    \item[\textbf{Step 3.}] For $i=0,\dots, 2L+2$, define $\tilde{\tau}_i$ by $\tilde{\sigma}_e(\tilde{\tau}_i)=\sigma_e(\tau_i)$, and, for $j=1,\dots, L$, replace $\tilde{\sigma}_e\vert_{[\tilde{\tau}_{2j-1},\tilde{\tau}_{2j+1}]}$ by the unique minimizing geodesic in $\Omega$ between $\tilde{\sigma}_e(\tilde{\tau}_{2j-1})$ and $\tilde{\sigma}_e(\tilde{\tau}_{2j+1})$. This defines a piecewise linear map $\sigma_o:I\to\Omega.$
    \item[\textbf{Step 4.}] Finally, reparametrize $\sigma_o$ to get a constant-speed curve $\tilde{\sigma}_o$.
\end{itemize}
Define
\begin{equation}
    \Psi(\sigma)\coloneqq\tilde{\sigma}_o.
\end{equation}

\subsubsection{Properties of the (fixed-endpoint) curve shortening map} The main properties of the map $\Psi$ are summarized in the following theorem, which follows directly from \cite[Section 3]{xinFB}, once we notice that, since the boundary of $\Omega$ is strictly convex, if the image of $\sigma\in X_{(L+1)r_0}$ lies in $\Omega$, then so does the image of $\Psi(\sigma)$.

\begin{theorem}\label{thm:BCSP} The (fixed-endpoint) curve shortening map $\Psi:X_{(L+1)r_0}\to\Lambda$ satisfies the following properties:
    \begin{enumerate}
        \item $\Psi$ is continuous (in the $W^{1,2}$ topology).
        \item For all $\sigma\in X_{(L+1)r_0}$, $\Psi(\sigma)$ is homotopic to $\sigma$ along maps $\{\sigma_t\}_{t\in [0,1]}$ in $X_{(L+1)r_0}$. 
        Also, for $t_2\geq t_1$, $\length(\sigma_{t_2})\leq\length(\sigma_{t_1})$.
        \item There is a continuous function $\phi:[0,+\infty)\to[0,+\infty)$, with $\phi(0)=0$, such that
        \begin{equation}
            \dist(\sigma,\Psi(\sigma))\leq\phi\left(\frac{\length(\sigma)^2-\length(\Psi(\sigma))^2}{\length(\Psi(\sigma))^2}\right).
        \end{equation}
        \item $\Psi(\sigma)=\sigma$ if and only if $\sigma\in G$.
        \item For any $\varepsilon>0$ there exists $\delta>0$ such that if $\sigma\in\Lambda$ and $\dist(\sigma, G)\geq \varepsilon$, then
        \begin{equation}
            \length(\Psi(\sigma))\leq\length(\sigma)-\delta.
        \end{equation}
    \end{enumerate}
\end{theorem}

\begin{remark}
    In the statement of the theorem, and for the rest of this section, $\dist(\cdot,\cdot)$ refers to distance in the $W^{1,2}$-norm. By the embedding of $W^{1,2}$ into $C^{0,\frac{1}{2}}$, it also controls the $C^0$-distance.
\end{remark}

\subsection{Sweepouts}

We now want to consider \textit{sweepouts} of $\Omega$ connecting $\gamma^1$ to $\gamma^2$.

First, pick a large constant $K\gg\max\{\E(\gamma^1)^{1/2}, \E(\gamma^2)^{1/2}\}$ and define the initial space $\cS$ of sweepouts as follows
\begin{equation}\label{eqn:space_S}
    \cS\coloneqq\{\Gamma\in C^0(I,X_K)\mid \Gamma(0,\cdot)=\gamma^1\text{ and }\Gamma(1,\cdot)=\gamma^2\}.
\end{equation}
As $\Omega$ is a topological disk in $S^2$, with boundary given by the union of $\gamma^1$ and $\gamma^2$, for large enough $K>0$, the set $\cS$ is clearly non-empty.
For each fixed $t\in I$, denote the map $\tau\to\Gamma(t,\tau)$ by $\Gamma_t$.
If $\Gamma\in\cS$, then $\max_{t\in[0,1]}\E(\Gamma_t)\leq K^2$, and, for all $t\in I, \tau_1,\tau_2\in I$, we have (by H\"{o}lder's inequality):
\begin{equation}
    \dist_{\Omega}(\Gamma_t(\tau_1), \Gamma_t(\tau_2))\leq |\tau_1-\tau_2|^{\frac{1}{2}}\E(\Gamma_t)^{\frac{1}{2}}\leq K |\tau_1-\tau_2|^{\frac{1}{2}}
\end{equation}
In particular, by choosing $L$ large enough, we can guarantee that $\dist_{\Omega}(\Gamma_t(\tau_1), \Gamma_t(\tau_2))<r_0$, if $|\tau_1-\tau_2|\leq\frac{1}{2L}$, and the argument in \cite[Lemma 4.1]{xinFB} guarantees that $\Psi\circ\Gamma:I\times I\to\Lambda$ is homotopic to $\Gamma$ via maps in $X_K$.

We can now refine $\cS$ by considering sweepouts in the class of piecewise linear maps.
\begin{equation}\label{eqn:space S_Lambda}
    \cS_\Lambda\coloneqq\{\Gamma\in C^0(I,\Lambda)\mid \Gamma(0,\cdot)=\Psi(\gamma^1)\text{ and }\Gamma(1,\cdot)=\Psi(\gamma^2)\}.
\end{equation}

\begin{remark}\label{rmk:X=PL}
Note that, by the previous observations, given $\Gamma\in\cS$, we can easily construct a corresponding sweepout $\tilde{\Gamma}\in\cS_\Lambda$, homotopic to $\Gamma$ via maps in $X_K$, by composing $\Gamma$ with the curve shortening map $\Psi$. Also, note that for all $t\in[0,1]$
\begin{equation}\label{eqn:X_to_PL}
    \length(\tilde{\Gamma}_t)\leq\length(\Gamma_t)
\end{equation}

Similarly, given $\Gamma\in\cS_\Lambda$, we can extend it to a sweepout $\hat{\Gamma}\in\cS$, homotopic to $\Gamma$ via maps in $X_K$, by concatenating a continuous length-nonincreasing homotopy between $\gamma^1$ and $\Psi(\gamma^1)$, the original sweepout $\Gamma$, and a continuous length-nondecreasing homotopy between $\Psi(\gamma^2)$ and $\gamma^2$\footnote{These continuous homotopies exist by Theorem \ref{thm:BCSP}.}. In this case,
\begin{equation}\label{eqn:PL_to_X}
    \max_{t\in[0,1]}\length(\hat{\Gamma}_t)=\max\{\length(\gamma^1), \max_{t\in[0,1]}\length(\Gamma_t), \length(\gamma^2)\}.
\end{equation}
\end{remark}

\subsection{Almost maximal implies almost critical} Using the machinery we have introduced, we can define the \textit{length-width} $W_0$ of $\Omega$ by 
\begin{equation}
    W_0=\inf_{\Gamma\in\cS}\max_{t\in [0,1]}\length(\Gamma_t)
\end{equation}
Note that $W_0\geq\max\{\length(\gamma^1), \length(\gamma^2)\}$.

From now on, assume that
\begin{equation}\label{eqn:mountain_pass}
    W_0>\max\{\length(\gamma^1), \length(\gamma^2)\}.   
\end{equation}

\begin{remark}
Note that we can define a `piecewise linear' length-width $W^\Lambda_0$ by 
\begin{equation}\label{eqn:PL_width}
    W^\Lambda_0=\inf_{\Gamma\in\cS_\Lambda}\max_{t\in[0,1]}\length(\Gamma_t).
\end{equation}
By \eqref{eqn:X_to_PL} and \eqref{eqn:PL_to_X}, $W_0=\max\{\length(\gamma^1), W^\Lambda_0, \length(\gamma^2)\}$. Hence, under the assumption that \eqref{eqn:mountain_pass} holds, 
\begin{equation}\label{eqn:X=PL}
    W_0=W^\Lambda_0.
\end{equation}
\end{remark}

Analogously to \cite[Theorem 5.6]{CM} and \cite[Theorem 1.6]{xinFB}, we have:
\begin{theorem}\label{thm:almost}
    Suppose \eqref{eqn:mountain_pass} is satisfied, then there exists a sequence $\{\Gamma^j\}\subset\cS$ of sweepouts such that $\max_{t\in [0,1]}\length(\Gamma^j_t)\to W_0$ as $j\to\infty$, and for any $\eps>0$, there exists $\delta>0$ such that if:
    \begin{itemize}
        \item $j>1/\delta$, and
        \item there is $t_0\in[0,1]$ such that $\length(\Gamma^j_{t_0})>W_0-\delta$,
    \end{itemize}
    then
    \begin{equation*}
        \dist(\Gamma^j_{t_0}, G)<\eps.
    \end{equation*}
\end{theorem}

\begin{proof}
Let $\{\Gamma^j\}\subset\cS$ be such that $\max_{t\in [0,1]}\length(\Gamma^j_t)\leq W_0+1/j$. By applying the curve shortening map, we obtain a new sequence $\{\tilde{\Gamma}^j_t\}\subset\cS_\Lambda$, and $\length(\tilde{\Gamma}^j_t)\leq\length(\Gamma^j_t)$. By \eqref{eqn:X=PL},
\[W_0=W^\Lambda_0\leq\max_{t\in[0,1]}\length(\tilde{\Gamma}^j_t)\leq\max_{t\in[0,1]}\length(\Gamma^j_t)\leq W_0+1/j,\]
so 
\[\lim_{j\to\infty}\max_{t\in[0,1]}\length(\tilde{\Gamma}^j_t)=W_0.\]

Assume now for a contradiction that $\{\tilde{\Gamma}^j_t\}$ does not satisfy the second part of the Theorem. Then we can find $\varepsilon_0>0$ and sequences $\delta_i\to 0$, $j_i>1/\delta_i$, $t_i\in [0,1]$ such that $\length(\tilde{\Gamma}^{j_i}_{t_i})>W_0-\delta_i$ but $\dist(\tilde{\Gamma}^{j_i}_{t_i}, G)\geq \eps_0$.

Note that 
\begin{equation}\label{eqn:width_ineq}
    W_0-\delta_i<\length(\tilde{\Gamma}^{j_i}_{t_i})=\length(\tilde{\Gamma}^{j_i}_{t_i})\leq\length(\Gamma^{j_i}_{t_i})\leq W_0+\frac{1}{j_i}\leq W_0+\delta_i
\end{equation}

Thus, by Property (3) in Theorem \ref{thm:BCSP}, $\dist(\tilde{\Gamma}^{j_i}_{t_i}, \Gamma^{j_i}_{t_i})\leq\frac{\eps_0}{2}$ and
\[\dist(\Gamma^{j_i}_{t_i},G)\geq\dist(\tilde{\Gamma}^{j_i}_{t_i},G)-\dist(\tilde{\Gamma}^{j_i}_{t_i}, \Gamma^{j_i}_{t_i})\geq\frac{\eps_0}{2},\]
thus contradicting Property (5) in Theorem \ref{thm:BCSP}, since \eqref{eqn:width_ineq} implies that $\length(\Gamma^{j_i}_{t_i})-\length(\tilde{\Gamma}^{j_i}_{t_i})\to 0$. 

By applying the procedure in Remark \ref{rmk:X=PL} to $\{\tilde{\Gamma}^j_t\}\subset\cS_\Lambda$, we obtain the desired sequence of sweepouts.
\end{proof}

Theorem \ref{thm:almost} proves that if $W_0>\max\{\length(\gamma^1), \length(\gamma^2)\}$, then:
\begin{itemize}
    \item we can construct sweepouts of $\Omega$ starting at $\gamma^1$ and ending at $\gamma^2$ with the property that a slice of nearly maximal length in the sweepout is close to a geodesic in $\Omega$ between $p$ and $q$;
    \item the width $W_0$ is the length of an immersed geodesic in $\Omega$ between $p$ and $q$.
\end{itemize}

Hence,
\begin{corollary}\label{cor:immersed_geodesic}
    Suppose $W_0>\max\{\length(\gamma^1), \length(\gamma^2)\}$, then there is an immersed geodesic $\tilde{\sigma}$ in $\Omega$ joining $p$ to $q$, such that
    \begin{equation}
        \length(\tilde{\sigma})=W_0.
    \end{equation}
\end{corollary}

\subsection{Sweepouts by embedded curves} Finally, we notice that we can upgrade from sweepouts consisting of immersed curves to sweepouts consisting of embedded curves. This improvement follows directly from \cite[Lemma 16, Lemma 17]{carlottodelellis}, which uses the results of \cite{ChambersLio}.

\begin{proposition}[\cite{carlottodelellis}, \cite{ChambersLio}]\label{prop:embeddings}
For any given $\eps>0$, the following holds: for every $\Gamma\in\cS$ there is $\overline{\Gamma}\in\cS$ such that for all $t\in[0,1]$
\begin{itemize}
    \item $\length(\overline{\Gamma}_t)\leq\length(\Gamma_t)+\eps$
    \item $\Gamma_t:(0,1)\to\Omega$ is an embedding.
\end{itemize}
\end{proposition}

By applying Proposition \ref{prop:embeddings} to a sequence of sweepouts satisfying the conclusions of Theorem \ref{thm:almost}, we obtain the following strengthening of Theorem \ref{thm:almost}.

\begin{theorem}\label{thm:almost+}
    Suppose \eqref{eqn:mountain_pass} is satisfied, then there exists a sequence $\{\Gamma^j\}\subset\cS$ of sweepouts consisting of embedded curves such that $\max_{t\in [0,1]}\length(\Gamma^j_s)\to W_0$ as $j\to\infty$, and for any $\eps>0$, there exists $\delta>0$ such that if:
    \begin{itemize}
        \item $j>1/\delta$, and
        \item there is $t_0\in[0,1]$ such that $\length(\Gamma^j_{t_0})>W_0-\delta$,
    \end{itemize}
    then
    \begin{equation}\label{eqn:almost_critical}
        \dist(\Gamma^j_{t_0}, G)<\eps.
    \end{equation}
\end{theorem}

We now argue as in the proof of \cite[Proposition 14]{carlottodelellis}. By \eqref{eqn:almost_critical}, there is a sequence $\sigma_i=\Gamma^{j_i}_{t_i}$ which converges in $X$ (and, thus, uniformly) to a (smooth) geodesic $\tilde{\sigma}$, of length $W_0$. Even though, a priori, $\tilde{\sigma}$ need not be embedded, by uniqueness of solutions for the geodesic equation, it can only self-intersect transversely. However, this cannot happen for a uniform limit of embedded curves. Hence, we have the following strengthening of Corollary \ref{cor:immersed_geodesic}:

\begin{corollary}\label{cor:embedded_geodesic}
    Suppose $W_0>\max\{\length(\gamma^1), \length(\gamma^2)\}$, then there is an embedded geodesic $\tilde{\sigma}$ in $\Omega$ joining $p$ to $q$, such that
    \begin{equation}
        \length(\tilde{\sigma})=W_0.
    \end{equation}
\end{corollary}

\bibliographystyle{amsalpha}
\bibliography{bib}

\end{document}